\documentclass{amsart}
\usepackage{amssymb}

\newtheorem{theorem}{Theorem}[section]
\newtheorem{lemma}[theorem]{Lemma}

\theoremstyle{definition}
\newtheorem{definition}[theorem]{Definition}

\newtheorem{remark}[theorem]{Remark}
\newtheorem{proposition}[theorem]{Proposition}
\newtheorem{corollary}[theorem]{Corollary}

\theoremstyle{remark}

\numberwithin{equation}{section}

\newcommand{\Rmnum}[1]{\expandafter\@slowromancap\romannumeral #1@}

\begin{document}

\title{On Random Convex Analysis}

\author{Tiexin Guo}
\address{School of Mathematics and Statistics, Central South University,
 Changsha 410083, China}
\email{tiexinguo@csu.edu.cn}
\thanks{This first four authors of this paper are supported by the National Natural Science Foundations No. 11171015 and No. 11571369.}

\author{Erxin Zhang}
\address{School of Mathematics and Statistics, Central South University,
 Changsha 410083, China}
\email{zhangerxin6666@163.com}

\author{Mingzhi Wu}
\address{School of Mathematics and Statistics, Central South University,
 Changsha 410083, China}
\email{wumingzhi@csu.edu.cn}

\author{Bixuan Yang}
\address{School of Mathematics and Statistics, Central South University,
 Changsha 410083, China}
\email{bixuanyang@126.com}

\author{George Yuan}
\address{Institute of Risk Management, Department of Mathematics,
 Tongji University, Shanghai, 200092, China; Centre for Financial Engineering,
  Soochow University, Suzhou, 215006, China }
\email{george\_yuan99@tonji.edu.cn; george\_yuan99@suda.edu.cn}
\thanks{George Yuan: Corresponding author}

\author{Xiaolin Zeng}
\address{School of Mathematics and Statistics, Chongqing Technology and Business University,Chongqing 400067, China}
\email{xlinzeng@163.com}
\thanks{The sixth author of this paper is supported by the National Natural Science Foundation No. 11301568.}





\begin{abstract}
Recently, based on the idea of randomizing space theory, random convex analysis has been being developed in order to deal with the corresponding problems in random environments such as analysis of conditional convex risk measures and the related variational problems and optimization problems. Random convex analysis is convex analysis over random locally convex modules. Since random locally convex modules have the more complicated topological and algebraic structures than ordinary locally convex spaces, establishing random convex analysis will encounter harder mathematical challenges than classical convex analysis so that there are still a lot of fundamentally important unsolved problems in random convex analysis. This paper is devoted to solving some important theoretic problems. First, we establish the inferior limit behavior of a proper lower semicontinuous $L^0$--convex function on a random locally convex module endowed with the locally $L^0$--convex topology, which makes perfect the Fenchel--Moreau duality theorem for such functions. Then, we investigate the relations among continuity, locally $L^0$--Lipschitzian continuity and almost surely sequent continuity of a proper $L^0$--convex function. And then, we establish the elegant relationships among subdifferentiability, G\^ateaux--differentiability and Fr\'ech\'et--differentiability for a proper $L^0$--convex function defined on random normed modules. At last, based on the Ekeland's variational principle for a proper lower semicontinuous $\bar{L}^0$--valued function, we show that $\varepsilon$--subdifferentials can be approximated by subdifferentials. We would like to emphasize that the success of this paper lies in simultaneously considering the $(\varepsilon, \lambda)$--topology and the locally $L^0$--convex topology for a random locally convex module.
\end{abstract}

\maketitle


.
\section{Introduction}\label{sec:1}
Classical convex analysis has played various kinds of important roles in Mathematical finance, variational problems, optimization problems, Banach space theory and nonlinear functional analysis, see, for instance,\cite {AB06,ET99,Phe89,Roc70,Yuan98,Yuan99}. In particular, since Artzner,et.al \cite{ADEH99} introduced coherent risk measures in 1999 and subsequently F\"{o}llmer and Schied \cite{FS02} and Fritlelli and Rosazza Gianin \cite{FRG02} independently introduced more general convex risk measures in 2002, classical convex analysis has become the analytic foundation for convex risk measures, please refer to \cite{FS11} for applications of convex analysis to risk measures and other topics in stochastic finance.

In conditional or dynamic setting, the notion of a conditional (or dynamic) convex risk measure is required to measure risk more precisely by making full use of information from markets or environment, which was independently introduced by Detlefsen and Scandolo \cite{DS05} and Bion--Nadal \cite{BN04} in 2004. At the outset, classical convex analysis still can develop its power for the study of conditional convex risk measures on bounded financial positions, see, for instance, \cite{DS05,BN04,FP06}. However, in general, classical convex analysis no longer applies to the study of conditional convex risk measures, for example, classical convex analysis can not deal with the dual representation of conditional convex risk measures defined on unbounded financial positions. Such a phenomenon was first pointed out by Filipovi\'c, Kupper and Vogelpoth in \cite{FKV09}, where locally $L^0$--convex modules were introduced and a hyperplane separation theorem between two $L^0$--convex sets with one of them open was established. Besides these important contributions, Filipovi\'c, Kupper and Vogelpoth in \cite{FKV09} also made an attempt to establish convex analysis over locally $L^0$--convex modules (called random convex analysis). However, locally $L^0$--convex modules have the complicated topological and algebraic structures, as pointed out by Guo, et.al in \cite{GZZ12} and attested, independently, by Zapata in \cite{Zap17} and Wu and Guo in \cite{WG15}, Filipovi\'c, et.al's paper \cite{FKV09} did not well deal with the complicated topological and algebraic structures so that random convex analysis  established in \cite{FKV09} was far from meeting the needs of conditional convex risk measures.

In fact, in \cite{GZZ12} we started a new approach to random convex analysis, namely choosing random locally convex modules as the space framework for random convex analysis. Although both random locally convex modules and locally $L^0$--convex modules are a random generalization of classical locally convex spaces, the structure of random locally convex modules is determined by a family of $L^0$--seminorms and the family of $L^0$--seminorms can simultaneously induce two kinds of topologies--the $(\varepsilon, \lambda)$--topology and the locally $L^0$--convex topology, whereas locally $L^0$--convex modules only involves the locally $L^0$--convex topology. Further, the two kinds of topologies have their respective advantages and disadvantages and in particular there are natural connections between basic theories derived from the two kinds of topologies for random locally convex modules, see \cite{Guo10,Guo13,GY12,GZZ15a,ZG12} for details, where readers can see that the advantages and disadvantages of the two kinds of topologies may complement each other. Thus random convex analysis can be thoroughly treated only when random convex analysis is put into random locally convex modules. Recently, we have developed some basic results of random convex analysis along the above--stated idea, for example, in \cite{GZZ15a} we gave the refined hyperplane separation theorem between a point and $L^0$--convex closed set and Fenchel--Moreau duality theorem for a proper lower semicontinuous $L^0$--convex function, in \cite{GZZ15b} we gave continuity and subdifferentiability theorems for a proper lower semicontinuous $L^0$--convex function on an $L^0$--pre--barrelled random locally convex module and in particular gave a characterization for a random locally convex module to be $L^0$--pre--barrelled, and in \cite{GZZ14} we gave some applications of random convex analysis to conditional risk measures. In fact, the work in \cite{GZZ15a,GZZ15b} has showed that establishing random convex analysis requires almost all achievements from random functional analysis (often also called random metric theory), which is concerned with analytics of random metric spaces, random normed modules and random inner product modules. Such a new approach to random functional analysis was initiated by Guo in \cite{Guo92,Guo93} who was motivated from the theory of probabilistic metric spaces \cite{SS8305} where K.Menger, B.Schweizer and A.Sklar advocated the idea of randomizing space theory.

There are many inherent or essential challenges in the course of the development of random convex analysis. In fact, these challenges also company random functional analysis all the time. Let $(\Omega, \mathcal{F}, P)$ be a probability space, $K$ the scalar field of real or complex numbers, $L^0(\mathcal{F}, K)$ the algebra of equivalence classes of $K$--valued random variables on $\Omega$ and $\bar{L}^0(\mathcal{F})$ the set of equivalence classes of extended real--valued random variables on $\Omega$. It is well known from \cite{DS57} that $\bar{L}^0(\mathcal{F})$ is a complete lattice under the partial order $\leq\ :\ \xi \leq \eta$ iff $\xi(\omega) \leq \eta(\omega)$ for $P$--almost surely all $\omega$ in $\Omega$ and $L^0(\mathcal{F}, R)$ is an order--complete lattice, where $R$ stands for the set of real numbers. The order $\leq$ on $L^0(\mathcal{F}, R)$ is a partial order,which brings huge difficulties to the study of random convex analysis since $L^0$--norms, $L^0$--seminorms and $L^0$--convex functions take their values in $L^0(\mathcal{F}, R)$ or $\bar{L}^0(\mathcal{F})$ unlike the usual norms, seminorms and convex functions with their values in $R$ or $[-\infty, +\infty]$. On the other hand, the useful topology for $R$ is often unique, namely the Euclidean topology, whereas $L^0(\mathcal{F}, R)$ possesses many useful topologies, for example, the topology of convergence in probability measure and the locally $L^0$--convex topology (somewhat similar to the topology of uniform convergence), which makes random normed modules and random locally convex modules possess very complicated topological structure. In addition, random normed modules and random locally convex modules are $L^0(\mathcal{F}, K)$--modules, unlike normed spaces and locally convex spaces as linear spaces over $K$, since $L^0(\mathcal{F}, K)$--modules have extremely complicated algebraic structure the study of random normed modules and random locally convex modules often requires the analysis of complicated stratification structure. It is because of the above--stated complications that there remain many basic and important problems unsolved in random convex analysis. This paper continues the study of random convex analysis and solves some basic theoretical problems.

In this paper, we first give a thorough treatment of lower semicontinuity for a proper $\bar{L}^0$--valued function defined on a random locally convex module. Given a random locally convex module $(E, \mathcal{P})$, we always denote by $\mathcal{T}_{\varepsilon,\lambda}$ and $\mathcal{T}_c$ the $(\varepsilon, \lambda)$--topology and the locally $L^0$--convex topology induced by the family $\mathcal{P}$ of $L^0$--seminorms, respectively. Now, let $(E, \mathcal{P})$ be a random locally convex module over the real number field $R$ with base $(\Omega, \mathcal{F}, P)$ such that $E$ and $\mathcal{P}$ both have the countable concatenation property, $f\ :\ E \to \bar{L}^0(\mathcal{F})$ a proper and local function and $epi(f)\ :\ = \{ (x, r) \in E \times L^0(\mathcal{F}, R)\ |\ f(x) \leq r \}$. Then the following statements are equivalent:\\
$(1)$ $\{ x \in E\ |\ f(x) \leq r \}$ is $\mathcal{T}_c$--closed for any $r \in L^0(\mathcal{F}, R)$;\\
$(2)$ $epi(f)$ is closed in $(E, \mathcal{T}_{\varepsilon,\lambda}) \times (L^0(\mathcal{F}, R), \mathcal{T}_{\varepsilon,\lambda})$;\\
$(3)$ $epi(f)$ is closed in $(E, \mathcal{T}_c) \times (L^0(\mathcal{F}, R), \mathcal{T}_c)$;\\
$(4)$ $\underline{lim}_{\alpha}f(x_{\alpha}) \geq f(x)$ for any $x \in E$ and any net $\{ x_{\alpha}, \alpha \in \Gamma \}$ convergent to $x$ with respect to $\mathcal{T}_c$, where $\underline{lim}_{\alpha}f(x_{\alpha}) = \bigvee_{\beta \in \Gamma}(\bigwedge_{\alpha \geq \beta}f(x_{\alpha}))$.

The proposition is crucial in random convex analysis. Although the previous papers from Filipovi\'c, Kupper and Vogelpoth's \cite{FKV09} to Guo, Zhao and Zeng's \cite{GZZ12} to the earlier version of our recent paper \cite{GZZ15a} are devoted to the study of equivalence among (1), (3) and (4), however, in \cite{GZZ15a} we realized that these papers had not yet given a strict proof that (1) implies (4), which is the most difficult part of the proof of the proposition. This paper, for the first time, completes the proof.

Then, this paper is devoted to the continuity problem. In \cite{GZZ15b}, Guo, Zhao and Zeng established continuity theorem for a proper $\mathcal{T}_c$--lower semicontinuous $L^0$--convex function. In this paper, we show that a proper $L^0$--convex function is locally $L^0$--Lipschitzian at this point if it is $\mathcal{T}_c$--continuous at some point, which further implies that it is also sequently continuous in the sense of almost sure convergence. In particular for an $L^0$--valued $L^0$--convex function defined on a complete random normed module with the countable concatenation property, we show that $\mathcal{T}_{\varepsilon,\lambda}$--continuity, $\mathcal{T}_c$--continuity  and almost surely sequent continuity coincide.

In \cite{GZZ15b}, Guo, Zhao and Zeng established subdifferentiability theorem for a proper lower semicontinuous $L^0$--convex function. In this paper, we further establish the operation laws on subdifferentials, in particular we also start the general study of G\^ateaux--and Fr\'ech\'et--differentiabilities for a proper $L^0$--convex function. It is well known that the theory of G\^ateaux--and Fr\'ech\'et--differentiability for ordinary convex functions is the deepest and the most difficult part in classical convex analysis, see \cite{Phe89}. However, for $L^0$--convex functions such a general theory has not been available except a few of study on Fr\'ech\'et--differentiability in the extremely special case \cite{CKV12}. Owing to the Riemann Calculus of Guo and Zhang on the abstract functions from an real interval to a $\mathcal{T}_{\varepsilon,\lambda}$--complete random normed module \cite{GZ12} we can present proper definitions of G\^ateaux--and Fr\'ech\'et--differentiabilities for a proper $L^0$--convex function defined on random normed modules and establish the elegant relationships among subdifferentiability, G\^ateaux--and Fr\'ech\'et--differentiabilities.

At last, by Guo and Yang's recently developed Ekeland's variational principle on complete random normed modules \cite{GY12} we show that $\varepsilon$--subgradients can be approximated by subgradients for a proper lower semicontinuous $L^0$--convex function in a simpler way than \cite{Yang12}, in particular we establish the property of G\^ateaux derivative at an approximate minimal point of a proper lower semicontinuous $\bar{L}^0$--valued function which is bounded below and G\^ateaux--differentiable on a complete random normed module, which is a generalization of the corresponding classical result of Ekeland in \cite{Eke74} from total order to partial order.

Finally, it should be pointed out that when we work on random convex analysis, we also has seen several other important works which are closely related to our work. For example, Frettelli and Maggis introduced conditionally evenly convex sets and studied the dual representation of conditionally evenly quasiconvex functions in \cite{FM14a,FM14b}, Eisele and Taieb studied weak topologies for locally convex $\lambda$--modules in \cite{ET15}, and Zapata gave randomized versions of Mazur lemma and Krein--\v{S}mulian theorem with application to conditional convex risk measures in \cite{Zap16}. Limited to space, we will investigate relations between these works and our work in the future papers.

The remainder of this paper is organized as follows : Section \ref{sec:2} is devoted to lower semicontinuity and Fenchel--Moreau duality for a proper $L^0$--convex function; Section \ref{sec:3} is devoted to the study of continuity for a proper $L^0$--convex function; in Section \ref{sec:4} we give the operation laws of subdifferentials for a proper $L^0$--convex function; in Section \ref{sec:5} we investigate G\^ateaux--and Fr\'ech\'et--differentiabilities for a proper $L^0$--convex function; in Section \ref{sec:6} we study the relation between subdifferentials and $\varepsilon$--subdifferentials for a proper lower semicontinuous $L^0$--convex function.

Throughout this paper, $(\Omega, \mathcal{F}, \mu)$ always denotes a given $\sigma$--finite measure space with $\mu(\Omega) > 0$, $K$ the scalar field $R$ of real numbers or $C$ of complex numbers, $L^0(\mathcal{F}, K)$ the algebra of equivalence classes of $K$--valued $\mathcal{F}$--measurable functions on $\Omega$ and $\bar{L}^0(\mathcal{F})$ the set of equivalence classes of extended real--valued $\mathcal{F}$--measurable functions on $\Omega$, where two functions are equivalent if they are equal almost everywhere (briefly, a.e.).

It is well known from \cite{DS57} that $\bar{L}^0(\mathcal{F})$ is an order complete lattice under the partial order : $\xi \leq \eta$ iff $\xi^0(\omega) \leq \eta^0(\omega)$ for almost all $\omega$ in $\Omega$, where $\xi^0$ and $\eta^0$ are arbitrarily chosen representatives of $\xi$ and $\eta$, respectively, further $\bigvee A$ and $\bigwedge A$ stand for the supremum and infimum of a subset $A$ of $\bar{L}^0(\mathcal{F})$, respectively. In addition, it is also well known that if $A$ is directed upwards (downwards) there exists a nondecreasing (nonincreasing) sequence $\{ a_n, n \in N \}$ $(\{ b_n, n \in N \})$ in $A$ such that $a_n \uparrow \bigvee A$ $(b_n \downarrow \bigwedge A)$. $\bar{L}^0(\mathcal{F})$ has the largest element and smallest element, denoted by $+\infty$ and $-\infty$, respectively, namely $+\infty$ and $-\infty$ stand for the equivalence classes of constant functions with values $+\infty$ and $-\infty$ on $\Omega$, respectively. Specially,  $L^0(\mathcal{F}, R)$ is order complete as a sublattice of $\bar{L}^0(\mathcal{F})$.

Let $A \in \mathcal{F}$ and $\xi$ and $\eta$ be in $\bar{L}^0(\mathcal{F})$, we say that $\xi > \eta$ on $A$ ($\xi \geq \eta$ on $A$) if $\xi^0(\omega) > \eta^0(\omega)$ (accordingly, $\xi^0(\omega) \geq \eta^0(\omega)$) for almost all $\omega \in A$, where $\xi^0$ and $\eta^0$ are arbitrarily chosen representatives of $\xi$ and $\eta$, respectively. Similarly, one can understand $\xi \neq \eta$ on $A$ and $\xi = \eta$ on $A$. Specially, $\tilde{I}_A$ stands for the equivalence class of $I_A$, where $I_A(\omega) = 1$ if $\omega \in A$, and 0 if $\omega \notin A$.

This paper always employs the following notation:

$L^0(\mathcal{F}) = L^0(\mathcal{F}, R)$;

$L^0_+(\mathcal{F}) = \{ \xi \in L^0(\mathcal{F})\ |\ \xi \geq 0 \}$;

$L^0_{++}(\mathcal{F}) = \{ \xi \in L^0(\mathcal{F})\ |\ \xi > 0$ on  $ \Omega  \}$.

Similarly, one can understand $\bar{L}^0_+(\mathcal{F})$ and  $\bar{L}^0_{++}(\mathcal{F})$.

\section{Lower semicontinuity and Fenchel--Moreau duality}\label{sec:2}

The main result in this section is Theorem \ref{the:2.13}. Let us first recapitulate some known terminology.

Let $E$ be a left module over the algebra $L^0(\mathcal{F}, K)$ (briefly, an $L^0(\mathcal{F}, K)$--module), the module multiplication $\xi \cdot x$ is simply denoted by $\xi x$ for any $\xi \in L^0(\mathcal{F}, K)$ and $x \in E$. A mapping $\|\cdot\| : E \to L^0_+(\mathcal{F})$ is called an $L^0$--seminorm on $E$ if it satisfies the following:\\
$(1)$ $\|\xi x\| = |\xi|\|x\|, \forall \xi \in L^0(\mathcal{F}, K)$ and $x \in E$;\\
$(2)$ $\|x+y\| \leq \|x\| + \|y\|, \forall x, y \in E$.

If, in addition, $\|x\| = 0$ implies $x = \theta$ (the null element of $E$), then $\|\cdot\|$ is called an $L^0$--norm on $E$, at this time the ordered pair $(E, \|\cdot\|)$ is called a random normed module (briefly, an $RN$ module) over $K$ with base $(\Omega, \mathcal{F}, \mu)$.

An ordered pair $(E, \mathcal{P})$ is called a random locally convex module (briefly, an $RLC$ module) over $K$ with base $(\Omega, \mathcal{F}, \mu)$ if $E$ is an $L^0(\mathcal{F}, K)$--module and $\mathcal{P}$ is a family of $L^0$--seminorms on $E$ such that $\bigvee \{ \|x\| : \|\cdot\| \in \mathcal{P} \} = 0$ implies $x = \theta$. Clearly, when $\mathcal{P}$ is a singleton consisting of an $L^0$--norm $\|\cdot\|$, an $RLC$ module $(E, \mathcal{P})$ becomes an $RN$ module $(E, \|\cdot\|)$, so the notion of an $RN$ module is a special case of that of an $RLC$ module.

Motivated by Schweizer and Sklar's work on random metric spaces and random normed linear spaces \cite{SS8305}, Guo introduced the notions of $RN$ modules and random inner product modules (briefly, $RIP$ modules) in \cite{Guo92,Guo93}. The importance of $RN$ modules lies in their $L^0(\mathcal{F}, K)$--module structure which makes $RN$ modules and their random conjugate spaces possess the same nice behaviors as normed spaces and their conjugate spaces. At almost the same time, Haydon, et.al also independently introduced the notion of an $RN$ module over the real number field $R$ with base a measure space (called randomly normed $L^0$--module in terms of \cite{HLR91}) as a tool for the study of ultrapowers of Lebesgue--Bochner function spaces. The notion of an $RLC$ module was first introduced by Guo and deeply developed by Guo and others in \cite{GP01,GXC09,GZ03}.

If $(\Omega, \mathcal{F}, \mu)$ is a finite measure space, we always use $\hat{\mu}$ for the probability defined by $\hat{\mu}(A) = \mu(A)/\mu(\Omega)$ for all $A \in \mathcal{F}$; if $(\Omega, \mathcal{F}, \mu)$ is a general $\sigma$--finite measure space, then we always use $\hat{\mu}$ for the probability defined by $\hat{\mu}(A) = \sum^{\infty}_{n=1}\frac{1}{2^n} \frac{\mu(A \cap \Omega_n)}{\mu(\Omega_n)}$ for all $A \in \mathcal{F}$, where $\{ \Omega_n : n \in N \}$ is a countable partition of $\Omega$ to $\mathcal{F}$ such that $0 < \mu(\Omega_n) < +\infty$ for all $n \in N$ and $N$ denotes the set of positive integers.

Given an $RLC$ module $(E, \mathcal{P})$ over $K$ with base $(\Omega, \mathcal{F}, \mu)$, we always denote by $\mathcal{P}(F)$ the family of finite nonempty subsets of $\mathcal{P}$. For each $Q \in \mathcal{P}(F)$, $\|\cdot\|_Q : E \to L^0_+(\mathcal{F})$ is the $L^0$-seminorm defined by $\|x\|_Q = \bigvee \{ \|x\| : \|\cdot\| \in Q \}$ for all $x \in E$. Now, we can speak of the  $(\varepsilon, \lambda)$--topology as follows:
\begin{proposition}\label{pro:2.1}
\cite{GP01,GZ03,GXC09}. Let $(E, \mathcal{P})$ be an $RLC$ module over $K$ with base $(\Omega, \mathcal{F}, \mu)$. For any positive numbers $\varepsilon$ and $\lambda$ with $0 < \lambda <1$ and for any $Q \in \mathcal{P}(F)$, let $N_{\theta}(Q, \varepsilon, \lambda) = \{ x \in E : \hat{\mu} \{ \omega \in \Omega\ |\ \|x\|_Q(\omega) < \varepsilon \} > 1- \lambda \}$. Then $\{ N_{\theta}(Q, \varepsilon, \lambda)\ |\ \varepsilon > 0, 0 < \lambda < 1$, and $Q \in \mathcal{P}(F) \}$ forms the local base at $\theta$ of some Hausdorff linear topology for $E$, called the $(\varepsilon, \lambda)$--topology induced by $\mathcal{P}$.
\end{proposition}
From now on, for any $RLC$ module $(E, \mathcal{P})$ we always use $\mathcal{T}_{\varepsilon,\lambda}$ for the $(\varepsilon, \lambda)$--topology for $E$ induced by $\mathcal{P}$. It is clear that the absolute value $|\cdot|$ is an $L^0$--norm on $L^0(\mathcal{F}, K)$. $\mathcal{T}_{\varepsilon,\lambda}$ induced by $|\cdot|$ is exactly the topology of convergence locally in measure, namely a sequence $\{ \xi_n : n \in N \}$ converges in $\mathcal{T}_{\varepsilon,\lambda}$ to $\xi$ in $L^0(\mathcal{F}, K)$ if and only if it converges in measure to $\xi$ on each $A \in \mathcal{F}$ such that $0 < \mu(A) < +\infty$. It is easy to check that $(L^0(\mathcal{F}, K), \mathcal{T}_{\varepsilon,\lambda})$ is a metrizable topological algebra and for an $RLC$ module $(E, \mathcal{P})$ over $K$ with base $(\Omega, \mathcal{F}, \mu)$. $(E, \mathcal{T}_{\varepsilon,\lambda})$ is a topological module over the topological algebra $(L^0(\mathcal{F}, K), \mathcal{T}_{\varepsilon,\lambda})$.

In 2009, Filipovi\'c, Kupper and Vogelpoth introduced another kind of topology for $L^0(\mathcal{F}, K)$ : let $\varepsilon$ belong to $L^0_{++}(\mathcal{F})$ and $U(\varepsilon) = \{ \xi \in L^0(\mathcal{F}, K)\ |\ |\xi| \leq \varepsilon \}$. A subset $G$ of $L^0(\mathcal{F}, K)$ is said to be $\mathcal{T}_c$--open if for each $g \in G$ there exists some $U(\varepsilon)$ such that $g + U(\varepsilon) \subset G$. Denote by $\mathcal{T}_c$ the family of $\mathcal{T}_c$--open subsets of $L^0(\mathcal{F}, K)$, then $(L^0(\mathcal{F}, K), \mathcal{T}_c)$ is a topological ring, namely the multiplication and addition operations on $L^0(\mathcal{F}, K)$ are both jointly continuous. Let $E$ be an $L^0(\mathcal{F}, K)$--module  and $\mathcal{T}$ a topology for $E$, then the topological space $(E, \mathcal{T})$ is called a topological $L^0$--module in \cite{FKV09} if $(E, \mathcal{T})$ is a topological module over the topological ring $(L^0(\mathcal{F}, K), \mathcal{T}_c)$, namely the module operations: the module multiplication operation and addition operation are both jointly continuous. In \cite{FKV09}, a topological $L^0$--module $(E, \mathcal{T})$ is called a locally $L^0$--convex module if $\mathcal{T}$ possesses a local base at $\theta$ whose each element is $L^0$--convex, $L^0$--absorbent and $L^0$--balanced, at which time $\mathcal{T}$ is also called a locally $L^0$--convex topology. Here , a subset $U$ of $E$ is said to be $L^0$--convex if $\xi x + (1 - \xi)y \in U$ for all $x, y \in U$ and $\xi \in L^0_+(\mathcal{F})$ such that $0 \leq \xi \leq 1$; $L^0$--absorbent if for each $x \in E$ there exists some $\eta \in L^0_{++}(\mathcal{F})$ such that $\xi x \in U$ for any $\xi \in L^0(\mathcal{F}, K)$ such that $|\xi| \leq \eta$; and $L^0$--balanced if $\xi x \in U$ for all $x \in U$ and all $\xi \in L^0(\mathcal{F}, K)$ such that $|\xi| \leq 1$. The work in \cite{FKV09} leads directly to the following:
\begin{proposition} \label{pro:2.2}
\cite{FKV09}. Let $(E, \mathcal{P})$ be an $RLC$ module over $K$ with base $(\Omega, \mathcal{F}, \mu)$. For any $\varepsilon \in L^0_{++}(\mathcal{F})$ and $Q \in \mathcal{P}(F)$, let $N_{\theta}(Q, \varepsilon) = \{ x \in E\ |\ \|x\|_Q \leq \varepsilon \}$. Then $\{ N_{\theta}(Q, \varepsilon)\ |\ Q \in \mathcal{P}(F), \varepsilon \in L^0_{++}(\mathcal{F}) \}$ forms a local base at $\theta$ of some Hausdorff locally $L^0$--convex topology, which is called the locally $L^0$--convex topology induced by $\mathcal{P}$.
\end{proposition}

From now on, for an $RLC$ module $(E, \mathcal{P})$, we always use $\mathcal{T}_c$ for the locally $L^0$--convex topology induced by $\mathcal{P}$. Recently, it is proved independently in \cite{WG15,Zap17} that the converse of Proposition \ref{pro:2.2} is no longer true, namely not every locally $L^0$--convex topology is necessarily induced by a family of $L^0$--seminorms.

For the sake of convenience, this paper needs the following:
\begin{definition}\label{def:2.3}
\cite{Guo10}. Let $E$ be an $L^0(\mathcal{F}, K)$--module and $G$ a subset of $E$. $G$ is said to have the countable concatenation property if for each sequence $\{ g_n: n \in N \}$ in $G$ and each countable partition $\{ A_n: n \in N \}$ of $\Omega$ to $\mathcal{F}$ there always exists $g \in G$ such that ${\tilde I}_{A_n}g = {\tilde I}_{A_n}g_n$ for each $n \in N$. If $E$ has the countable concatenation property, $H_{cc}(G)$ denotes the countable concatenation hull of $G$, namely the smallest set containing $G$ and having the countable concatenation property.
\end{definition}

\begin{remark}\label{rem:2.4}
As pointed out in \cite{Guo10}, when $(E, \mathcal{P})$ is an $RLC$ module, g in Definition \ref{def:2.3}
must be unique, at which time we can write $g = \sum^{\infty}_{n=1}\tilde{I}_{A_n}g_n$.
\end{remark}

In \cite{FKV09}, a family $\mathcal{P}$ of $L^0$--seminorms on an $L^0(\mathcal{F}, K)$--module is said to have the countable concatenation property if each $L^0$--seminorm $\|\cdot\| : = \sum^{\infty}_{n=1}\tilde{I}_{A_n}\|\cdot\|_{Q_n}$ still belongs to $\mathcal{P}$ for each countable partition $\{ A_n : n \in N \}$ of $\Omega$ to $\mathcal{F}$ and each sequence $\{ Q_n : n \in N \}$ in $\mathcal{P}(F)$. We always denote $\mathcal{P}_{cc} = \{ \sum^{\infty}_{n=1}\tilde{I}_{A_n}\cdot\|\cdot\|_{Q_n} : \{ A_n : n \in N\}$ is a countable partition of $\Omega$ to $\mathcal{F}$ and $\{ Q_n : n \in N \}$ a sequence of $\mathcal{P}(F) \}$, called the countable concatenation hull of $\mathcal{P}$. Clearly, $\mathcal{P}$ has the countable concatenation property iff $\mathcal{P}_{cc} = \mathcal{P}$.

In random functional analysis, the notion of random conjugate spaces is crucial, which is defined as follows:

\begin{definition}\label{def:2.5}
\cite{Guo10}. Let $(E, \mathcal{P})$ be an $RLC$ module over $K$ with base $(\Omega, \mathcal{F}, \mu)$. Denote by $(E, \mathcal{P})^{\ast}_{\varepsilon,\lambda}$ the $L^0(\mathcal{F}, K)$--module of continuous module homomorphisms from $(E, \mathcal{T}_{\varepsilon,\lambda})$ to $(L^0(\mathcal{F}, K), \mathcal{T}_{\varepsilon,\lambda})$, called the random conjugate space of $(E, \mathcal{P})$ under $\mathcal{T}_{\varepsilon,\lambda}$; denote by $(E, \mathcal{P})^{\ast}_c$ the $L^0(\mathcal{F}, K)$--module of continuous module homomorphisms from $(E, \mathcal{T}_c)$ to $(L^0(\mathcal{F}, K), \mathcal{T}_c)$, called the random conjugate space of $(E, \mathcal{P})$ under $\mathcal{T}_c$.
\end{definition}

From now on, when $\mathcal{P}$ is understood, we often briefly write $E^{\ast}_{\varepsilon,\lambda}$ for $(E, \mathcal{P})^{\ast}_{\varepsilon,\lambda}$ and $E^{\ast}_c$ for $(E, \mathcal{P})^{\ast}_c$. When $\mathcal{P}$ has the countable concatenation property, it is proved in \cite{Guo10} that $E^{\ast}_{\varepsilon,\lambda} = E^{\ast}_c$. In general, $E^{\ast}_c \subset E^{\ast}_{\varepsilon,\lambda}$ and $E^{\ast}_{\varepsilon,\lambda}$ has the countable concatenation property. Recently, in \cite{GZZ15a} Guo, et.al established the following precise relation between $E^{\ast}_{\varepsilon,\lambda}$ and $E^{\ast}_c$.

\begin{proposition}\label{pro:2.6}
\cite{GZZ15a}. Let $(E, \mathcal{P})$ be an $RLC$ module. Then $E^{\ast}_{\varepsilon,\lambda} = H_{cc}(E^{\ast}_c)$.
\end{proposition}

\begin{remark}\label{rem:2.7}
For an $RLC$ module $(E, \mathcal{P})$, since $\mathcal{P}$ and $\mathcal{P}_{cc}$ induce the same $(\varepsilon, \lambda)$--topology on $E$, then $(E, \mathcal{P})^{\ast}_{\varepsilon,\lambda} = (E, \mathcal{P}_{cc})^{\ast}_{\varepsilon,\lambda}$. Since $\mathcal{P}_{cc}$ has the countable concatenation property, $(E, \mathcal{P}_{cc})^{\ast}_{\varepsilon,\lambda} = (E, \mathcal{P}_{cc})^{\ast}_c$, in fact, Proposition \ref{pro:2.6} has showed that $E^{\ast}_{\varepsilon,\lambda}= (E, \mathcal{P}_{cc})^{\ast}_c = H_{cc}(E^{\ast}_c)$!
\end{remark}

To state and prove the main result of this section, we still need Lemmas \ref{lem:2.8} and \ref{lem:2.9} below.

\begin{lemma}\label{lem:2.8}
\cite{Guo10}. Let $(E, \mathcal{P})$ be an $RLC$ module with base $(\Omega, \mathcal{F}, \mu)$ and $G \subset E$ such that $G$ has the countable concatenation property. Then $\bar{G}_{\varepsilon,\lambda} = \bar{G}_c$, where $\bar{G}_{\varepsilon,\lambda}$ and $\bar{G}_c$ stand for the closures of $G$ under $\mathcal{T}_{\varepsilon,\lambda}$ and $\mathcal{T}_c$, respectively.
\end{lemma}

\begin{lemma}\label{lem:2.9}
\cite{GZZ15a}. Let $(E, \mathcal{P})$ be an $RLC$ module such that $\mathcal{P}$ has the countable concatenation property, $M$ a $\mathcal{T}_c$--closed subset with the countable concatenation property and $x \in E$ such that $\tilde{I}_{A}\{x\} \bigcap \tilde{I}_{A}M = \emptyset$ for all $A \in \mathcal{F}$ with $\mu(A) > 0$. Then there is an $L^0$--convex, $L^0$--absorbent and $L^0$--balanced $\mathcal{T}_c$--neighborhood $U$ of $\theta$ such that $\tilde{I}_{A}(x + U) \bigcap \tilde{I}_{A}(M + U) = \emptyset$ for all $A \in \mathcal{F}$ with $\mu(A) > 0$.
\end{lemma}

\begin{remark}\label{rem:2.10}
The paper \cite{GZZ15a} provides a counterexample showing that if $M$ is merely a $\mathcal{T}_c$--closed subset such that $\tilde{I}_{A}M + \tilde{I}_{A^c}M \subset M$ for all $A \in \mathcal{F}$ then Lemma \ref{lem:2.9} is not necessarily true.
\end{remark}

\begin{remark}\label{rem:2.11}
In Lemma \ref{lem:2.9}, if for some $\mathcal{F}$-measurable subset $B$ such that $\mu(B) >0$ and $\tilde{I}_{A}\{x\} \bigcap \tilde{I}_{A}M = \emptyset$ for all $A \in \mathcal{F}$ such that $A \subset B$ and $\mu(A) > 0$, then we still can have that there is an $L^0$--convex, $L^0$--absorbent and $L^0$--balanced $\mathcal{T}_c$--neighborhood $U$ of $\theta$ such that $\tilde{I}_{A}(x + U) \bigcap \tilde{I}_{A}(M + U) = \emptyset$ for all $A \in \mathcal{F}$ with $A \subset B$ and $\mu(A) > 0$. In fact, let $E_B = \tilde{I}_{B}E : = \{ \tilde{I}_{B}x : x \in E\}$, $\mathcal{P}_B = \{ \|\cdot\||_{E_B} : \|\cdot\| \in \mathcal{P} \}$, where $\|\cdot\||_{E_B}$ stands for the restriction of $\|\cdot\|$ to $E_B$, then $(E_B, \mathcal{P}_B)$ is still an $RLC$ module with base $(B, B \bigcap \mathcal{F}, \mu_B)$, where $B \bigcap \mathcal{F} = \{ B \bigcap A : A \in \mathcal{F} \}$ and $\mu_B(B \bigcap A) = \mu (B \bigcap A)$ for all $A \in \mathcal{F}$, further in \cite{GZZ15b} Guo, et.al have showed that if $G$ is $\mathcal{T}_c$--open (or, $\mathcal{T}_c$--closed) in $(E, \mathcal{T}_c)$ such that $\tilde{I}_{B}G + \tilde{I}_{B^c}G \subset G$ then $\tilde{I}_{B}G$ is also $\mathcal{T}_c$--open (accordingly, $\mathcal{T}_c$--closed) in $(E_B, \mathcal{P}_B)$. Thus one can see this point by considering the relative topology.
\end{remark}

Let $(E, \mathcal{P})$ be an $RLC$ module over $R$ with base $(\Omega, \mathcal{F}, \mu)$ and $f : E \to \bar{L}^0(\mathcal{F})$. $dom(f) : = \{ x \in E\ |\ f(x) < +\infty$ on $\Omega$ \}, called the effective domain of $f$, and $epi(f) : = \{ (x, r) \in E \times L^0(\mathcal{F})\ |\ f(x) \leq r \}$, called the epigraph of $f$. $f$ is said to be local if $\tilde{I}_{A}f(x) = \tilde{I}_{A}(f(\tilde{I}_{A}x))$ for all $A \in \mathcal{F}$ and all $x \in E$; $f$ is said to be proper if $dom(f)$ is nonempty and $f(x) > -\infty$ on $\Omega$ for all $x \in E$; $f$ is said to be $L^0$--convex if $f(\xi x + (1 - \xi)y) \leq \xi f(x) + (1 - \xi)f(y)$ for all $x, y \in E$ and $\xi \in L^0_+(\mathcal{F})$ such that $0 \leq \xi \leq 1$. Here, we always adopt the convention that $0 \cdot (\pm \infty) = 0$ and $\infty - \infty = \infty$ (namely $+\infty + (-\infty) = +\infty$). It is proved in \cite{FKV09} that an $L^0$--convex function must be local and a proper and local function $f : E \to \bar{L}^0(\mathcal{F})$ is $L^0$--convex iff $epi(f)$ is $L^0$--convex in $E \times L^0(\mathcal{F})$.

\begin{definition}\label{def:2.12}
Let $(E, \mathcal{P})$ be an $RLC$ module over $R$ with base $(\Omega, \mathcal{F}, \mu)$ and $f : E \to \bar{L}^0(\mathcal{F})$ a proper function. $f$ is said to be $\mathcal{T}_{\varepsilon,\lambda}$--lower semicontinuous if $epi(f)$ is closed in $(E, \mathcal{T}_{\varepsilon,\lambda}) \times (L^0(\mathcal{F}), \mathcal{T}_{\varepsilon,\lambda})$; $f$ is said to be $\mathcal{T}_c$--lower semicontinuous if $\{ x \in E\ |\ f(x) \leq r \}$ is $\mathcal{T}_c$--closed for all $r \in L^0(\mathcal{F})$.
\end{definition}

Now, we can state and prove the main result of this section.

\begin{theorem}\label{the:2.13}
Let $(E, \mathcal{P})$ be an $RLC$ module over $R$ with base $(\Omega, \mathcal{F}, \mu)$ such that both $E$ and $\mathcal{P}$ have the countable concatenation property, $f : E \to \bar{L}^0(\mathcal{F})$ a proper and local function, then the following are equivalent:\\
$(1)$ $f$ is $\mathcal{T}_c$--lower semicontinuous;\\
$(2)$ $f$ is $\mathcal{T}_{\varepsilon,\lambda}$--lower semicontinuous;\\
$(3)$ $epi(f)$ is closed in $(E, \mathcal{T}_c) \times (L^0(\mathcal{F}), \mathcal{T}_c)$;\\
$(4)$ $\underline{lim}_{\alpha}f(x_{\alpha}) \geq f(x)$ for any $x \in E$ and any net $\{ x_{\alpha}, \alpha \in \Gamma \}$ convergent to $x$ with respect to $\mathcal{T}_c$, where $\underline{lim}_{\alpha}f(x_{\alpha}) = \bigvee_{\beta \in \Gamma}(\bigwedge_{\alpha \geq \beta}f(x_{\alpha}))$.
\end{theorem}

\begin{proof}
$(4)$ $\Rightarrow$ (3) $\Rightarrow$ (1) is clear. By definition and Lemma \ref{lem:2.8}, (2) $\Leftrightarrow$ (3) is also clear since $epi(f)$ has the countable concatenation property. The most difficult part of the proof is (1) $\Rightarrow$ (4) as follows.

For any fixed $x \in E$, let $\{ x_{\alpha}, \alpha \in \Gamma \}$ converge to $x$ with respect to $\mathcal{T}_c$. Further, let $r$ be any element of $L^0(\mathcal{F})$ such that $f(x) > r$ on $\Omega$.

Denote $\mathcal{A}_r = \{ A \in \mathcal{F}\ |\ \mu(A) > 0$ and there exists $z \in E$ such that $\tilde{I}_{A}f(z) \leq \tilde{I}_{A}r \}$. If $\mathcal{A}_r$ is empty, then for all $z \in E$ we always $f(z) > r$ on $\Omega$, which clearly means $\underline{lim}_{\alpha}f(x_{\alpha}) \geq r$. We will consider the case in which $\mathcal{A}_r$ is not empty as follows.

If $\mathcal{A}_r$ is not empty, then it must be directed upwards : in fact, let $A$ and $B \in \mathcal{A}_r$, then there exist $z_1$ and $z_2 \in E$ such that $\tilde{I}_{A}f(z_1) \leq \tilde{I}_{A} r$ and $\tilde{I}_{B}f(z_2) \leq \tilde{I}_{B}r$, according to the local property of $f$ one can have that
\begin{align}
  \tilde{I}_{A \bigcup B}f(\tilde{I}_{A}z_1 + \tilde{I}_{B \backslash A}z_2)
 =& (\tilde{I}_{A} + \tilde{I}_{B \backslash A})f(\tilde{I}_{A}z_1 + \tilde{I}_{B \backslash A}z_2) \notag\\
 =& \tilde{I}_{A}f(z_1) + \tilde{I}_{B \backslash A}f(z_2) \notag \\
\leq& \tilde{I}_{A \bigcup B} r \notag
\end{align}

Denote $A_r = esssup \mathcal{A}_r$ (namely, the essential supremum of $\mathcal{A}_r$, see, for instance, \cite{Guo10} for the notion of essential supremum), we will prove $A_r \in \mathcal{A}_r$ as follows. Since there exists a sequence $\{ A_n\ |\ n \in N \}$ in $\mathcal{A}_r$ such that $A_n \uparrow A_r$, correspondingly, there exists a sequence $\{ z_n\ |\ n \in N \}$ in $E$ such that $\tilde{I}_{A_n}f(z_n) \leq \tilde{I}_{A_n}r$ for each $n \in N$. Since $E$ has the countable concatenation property, there exists $z \in E$ such that $z = \sum^{\infty}_{n=1}\tilde{I}_{A_{n} \backslash A_{n-1}}z_n + \tilde{I}_{A^{c}_{r}}0$, where $A_0 = \emptyset$, further, we have that
\begin{align} \label{gongshi:xing}
   \tilde{I}_{A_r}f(z)
   =& (\sum^{\infty}_{n=1}\tilde{I}_{A_{n} \backslash A_{n-1}})f(z) \notag\\
   =& \sum^{\infty}_{n=1}\tilde{I}_{A_{n} \backslash A_{n-1}} f(\tilde{I}_{A_{n} \backslash A_{n-1}} z) \notag\\
   =& \sum^{\infty}_{n=1}\tilde{I}_{A_{n} \backslash A_{n-1}} f(z_n) (by\ the\ local\ property\ of\ f ) \notag\\
   \leq& \tilde{I}_{A_r}r \tag{$\ast$}
\end{align}
This shows that $A_r \in \mathcal{A}_r$.

Let $V = \{ z \in E\ |\ \tilde{I}_{A_r}f(z) \leq \tilde{I}_{A_r}r \}$. Since $dom(f) \neq \emptyset$, let $u$ be any fixed element of $dom(f)$, then we have that $V = \{ z \in E\ |\ \tilde{I}_{A_r}f(z) + \tilde{I}_{A^c_r}f(u) \leq \tilde{I}_{A_r} r + \tilde{I}_{A^c_r}f(u) \}$ = $\{ z \in E\ |\ f(\tilde{I}_{A_r}z + \tilde{I}_{A^c_r}u) \leq \tilde{I}_{A_r} r + \tilde{I}_{A^c_r}f(u) \}$ = $\{ z \in E\ |\ \tilde{I}_{A_r}z \in V(f, \xi) - \tilde{I}_{A^c_r}u \}$, where $V(f, \xi) = \{ x \in E\ |\ f(x) \leq \xi \}$ and $\xi = \tilde{I}_{A_r}r + \tilde{I}_{A^c_r}f(u)$. Since $f$ is $\mathcal{T}_c$--lower semicontinuous, $V(f, \xi)$ is $\mathcal{T}_c$--closed, which obviously implies that $V$ is also $\mathcal{T}_c$--closed.

Of course, $V$ also has the countable concatenation property since $V(f, \xi)$ possesses the property. We will verify that $\tilde{I}_{A}\{x\} \bigcap \tilde{I}_{A}V = \emptyset$ for all $A \in \mathcal{F}$ with $A \subset A_r$ and $\mu(A) > 0$ as follows.

In fact, if there exists $A \in \mathcal{F}$ with $A \subset A_r$ and $\mu(A) > 0$ such that $\tilde{I}_{A}x = \tilde{I}_{A}\cdot z_A$ for some $z_A \in V$, then by the local property of $f$, $\tilde{I}_{A}f(x) = \tilde{I}_{A}f(\tilde{I}_{A}x) = \tilde{I}_{A}f(\tilde{I}_{A}z_A) = \tilde{I}_{A}f(z_A) = \tilde{I}_{A} \cdot \tilde{I}_{A_r}f(z_A) \leq \tilde{I}_{A} \cdot \tilde{I}_{A_r}r = \tilde{I}_{A}r$, which contradicts the hypothesis on $r$.

Thus, by Remark \ref{rem:2.11} there exists some $\mathcal{T}_c$--neighborhood $U$ of $\theta$ such that $\tilde{I}_{A}(x + U) \bigcap \tilde{I}_{A}(V + U) = \emptyset$ for all $A \in \mathcal{F}$ with $A \subset A_r$ and $\mu(A) > 0$. Since $\{ x_{\alpha}, \alpha \in \Gamma \}$ converges to $x$, there is some $\alpha_{0} \in \Gamma$ such that $x_{\beta} \in x + U$ for all $\beta \geq \alpha_{0}$. Then, for all $\beta \geq \alpha_{0}$ and all $A \in \mathcal{F}$ with $A \subset A_r$ and $\mu(A) > 0$, we must have that $\tilde{I}_{A}x_{\beta} \notin \tilde{I}_{A}V$, which means that $\tilde{I}_{A_r}f(x_{\beta}) > \tilde{I}_{A_r}r$ on $A_r$. In fact, if there are some $\beta \geq \alpha_{0}$ and some $A \in \mathcal{F}$ with $A \subset A_r$ and $\mu(A) > 0$ such that $\tilde{I}_{A} \cdot \tilde{I}_{A_r}f(x_\beta) \leq \tilde{I}_{A} \cdot \tilde{I}_{A_r}r$, since by (\ref{gongshi:xing}) we also have that $\tilde{I}_{A_r \backslash A}f(z) \leq \tilde{I}_{A_r \backslash A}r$, where $z$ is as given in (\ref{gongshi:xing}), to sum up, we can get that $\tilde{I}_{A_r}f(\tilde{I}_{A}x_\beta + \tilde{I}_{A_r \backslash A}z) \leq \tilde{I}_{A_r}r$, namely $\tilde{I}_{A}x_\beta + \tilde{I}_{A_r \backslash A}z \in V$. Then $\tilde{I}_{A}x_\beta = \tilde{I}_{A}(\tilde{I}_{A}x_\beta + \tilde{I}_{A_r \backslash A}z) \in \tilde{I}_{A}V$, which contradicts the fact that $\tilde{I}_{A}x_\beta \notin \tilde{I}_{A}V$.

Finally, by the definition of $\mathcal{A}_r$, $f(y) > r$ on $A^c_r$ for all $y \in E$, and hence $f(x_\beta) = \tilde{I}_{A_r}f(x_\beta) + \tilde{I}_{A^c_r}f(x_\beta) > \tilde{I}_{A_r}r + \tilde{I}_{A^c_r}r = r$ on $\Omega$ for all $\beta \geq \alpha_0$, which further implies that $\underline{lim}_{\alpha} f(x_\alpha) \geq r$.

Up to now, we have proved that $\underline{lim}_{\alpha} f(x_\alpha) \geq r$ in either case in which $\mathcal{A}_r$ is empty or nonempty. Since $r$ is an arbitrarily chosen element such that $f(x) > r$ on $\Omega$, we have $\underline{lim}_{\alpha} f(x_\alpha) \geq f(x)$.

This completes the proof
\end{proof}

\begin{remark}\label{lem:2.14}
Let $(E, \mathcal{P})$ be an $RLC$ module over $R$ with base $(\Omega, \mathcal{F}, \mu)$ and $f : E \to \bar{L}^0(\mathcal{F})$ a proper and local function. Let us observe the following three statements:\\
$(1)^\prime$ $\{ x \in E\ |\ f(x) \leq r \}$ is $\mathcal{T}_{\varepsilon,\lambda}$--closed for all $r \in L^0(\mathcal{F})$;\\
$(2)^\prime$ $epi(f)$ is closed in $(E, \mathcal{T}_{\varepsilon,\lambda}) \times (L^0(\mathcal{F}), \mathcal{T}_{\varepsilon,\lambda})$;\\
$(3)^\prime$ $\underline{lim}_{\alpha} f(x_\alpha) \geq f(x)$ for all $x \in E$ and all net $\{ x_\alpha, \alpha \in \Gamma \}$ convergent to $x$ with respect to $\mathcal{T}_{\varepsilon,\lambda}$.\\
Generally, one always has $(3)^\prime$ $\Rightarrow$ $(2)^\prime$ $\Rightarrow$ $(1)^\prime$. Although under the assumption of Theorem \ref{the:2.13} one also has $(1)^\prime$ $\Leftrightarrow$ $(2)^\prime$, $(1)^\prime$ or $(2)^\prime$ did not imply $(3)^\prime$, in fact, we can construct examples showing that $(3)^\prime$ is not necessarily true even for a $\mathcal{T}_{\varepsilon,\lambda}$--continuous function. Generally, we do not know if $(1)^\prime$ implies $(2)^\prime$, either, but $(2)^\prime$ as the definition of a $\mathcal{T}_{\varepsilon,\lambda}$--lower semicontinuous function has met the needs of the study of $\mathcal{T}_{\varepsilon,\lambda}$--lower semicontinuous functions, see, for instance, \cite{GY12,GZZ15a}. Whereas, for the case of the locally $L^0$--convex topology $\mathcal{T}_c$, since we are often forced to assume that $RLC$ module $(E, \mathcal{P})$ in question satisfy the hypothesis of Theorem \ref{the:2.13}, we adopt the general definition of a $\mathcal{T}_c$--lower semicontinuous function as Definition \ref{def:2.12}. But, in \cite{GZZ15a} we used (3) of Theorem \ref{the:2.13} as the notion of a $\mathcal{T}_c$--lower semicontinuous function since we then did not know if (1) of Theorem \ref{the:2.13} indeed implies (4), and hence also (3), of Theorem \ref{the:2.13}.
\end{remark}

Now, owing to Theorem \ref{the:2.13} we can make perfect the Fenchel--Moreau duality theorem for a $\mathcal{T}_c$--lower semicontinuous function. From now on, we always use Definition \ref{def:2.12} as the notion of a $\mathcal{T}_{\varepsilon,\lambda}$-- or $\mathcal{T}_c$--lower semicontinuous function unless otherwise stated.

Let $(E, \mathcal{P})$ be an $RLC$ module over $R$ with base $(\Omega, \mathcal{F}, \mu)$ and $f : E \to \bar{L}^0(\mathcal{F})$. Define $f^{\ast}_{\varepsilon,\lambda} : E^{\ast}_{\varepsilon,\lambda} \to \bar{L}^0(\mathcal{F}),\ f^{\ast \ast}_{\varepsilon,\lambda} : E \to \bar{L}^0(\mathcal{F}),\ f^{\ast}_c : E^{\ast}_c \to \bar{L}^0(\mathcal{F})$ and $f^{\ast \ast}_c : E \to \bar{L}^0(\mathcal{F})$ as follows:\\
$f^{\ast}_{\varepsilon,\lambda}(g) = \bigvee \{ g(x) - f(x)\ |\ x \in E \}$ for all $g \in E^{\ast}_{\varepsilon,\lambda}$;\\
$f^{\ast \ast}_{\varepsilon,\lambda}(x) = \bigvee \{ g(x)- f^{\ast}_{\varepsilon,\lambda}(g)\ |\ g \in E^{\ast}_{\varepsilon,\lambda} \}$ for all $x \in E$;\\
$f^{\ast}_{c}(g) = \bigvee \{ g(x) - f(x)\ |\ x \in E \}$ for all $g \in E^{\ast}_{c}$;\\
$f^{\ast \ast}_{c}(x) = \bigvee \{ g(x) - f^{\ast}_{c}(g)\ |\ g \in E^{\ast}_{c} \}$ for all $x \in E$.

In \cite{GZZ15a}, Guo, et.al proved Proposition \ref{pro:2.15} below:
\begin{proposition}\label{pro:2.15}
\cite{GZZ15a}. Let $(E, \mathcal{P})$ be an $RLC$ module over $R$ with base $(\Omega, \mathcal{F}, \mu)$ and $f : E \to \bar{L}^0(\mathcal{F})$ a proper $\mathcal{T}_{\varepsilon,\lambda}$--lower semicontinuous $L^0$--convex function. Then $f^{\ast \ast}_{\varepsilon,\lambda} = f$.
\end{proposition}

In the proof of Proposition \ref{pro:2.15} the paper \cite{GZZ15a} used a technique, namely Lemma \ref{lem:2.16} below, however in \cite{GZZ15a} some details were omitted, we will give a detailed proof of Lemma \ref{lem:2.16} since those omitted details will be used in this paper.

\begin{lemma}\label{lem:2.16}
Let $(E, \mathcal{P})$ be an $RLC$ module over $R$ with base $(\Omega, \mathcal{F}, \mu)$ and $f : E \to \bar{L}^0(\mathcal{F})$ a proper $\mathcal{T}_{\varepsilon,\lambda}$--lower semicontinuous $L^0$--convex function. If $x_0 \in E$ and $\beta \in L^0(\mathcal{F})$ are such that $f(x_0) > \beta$ on $\Omega$, then there is an $\mathcal{T}_{\varepsilon,\lambda}$--continuous affine function $h = g + \alpha$ $($where $g \in E^{\ast}_{\varepsilon,\lambda}$ and $\alpha \in L^0(\mathcal{F}))$ such that $h(x_0) = g(x_0) + \alpha = \beta$ and $h(x) \leq f(x)$ for all $x \in E$.
\end{lemma}

Before the proof of Lemma \ref{lem:2.16} we first give the two separation propositions, namely Proposition \ref{pro:2.17} and Corollary \ref{cor:2.18} below since the proof of Lemma \ref{lem:2.16} is based on Proposition \ref{pro:2.17}.

Let $(E, \mathcal{P})$ be an $RLC$ module over $K$ with base $(\Omega, \mathcal{F}, \mu)$, $x \in E$ and $M \subset E$. Let $d_{Q}(x, M)= \bigwedge \{ \|x - y\|_{Q}\ |\ y \in M \}$ for all $Q \in \mathcal{P}(F)$, and $d(x, M) = \bigvee_{Q \in \mathcal{P}(F)} d_{Q}(x, M)$. For any representative $d^{0}(x, M)$ of $d(x, M)$, we always use $(d(x, M) > 0)$ for $\{ \omega \in \Omega\ |\ d^{0}(x, M)(\omega) > 0 \}$. Just as pointed out in \cite{GZZ15a}, if $M$ is a $\mathcal{T}_{\varepsilon,\lambda}$--closed $L^0$--convex set and $x \notin M$ then $\tilde{I}_{A}\{x\} \cap \tilde{I}_{A}M = \emptyset$ for all $A \in \mathcal{F}$ with $A \subset (d(x, M) > 0)$ and $\mu(A) > 0$, in particular $\tilde{I}_{A}\{x\} \bigcap \tilde{I}_{A}M = \emptyset$ for all $A \in \mathcal{F}$ with $\mu(A) > 0$ iff $\mu(\Omega \backslash (d(x, M) > 0)) = 0$ (namely, $d(x ,M) > 0$ on $\Omega$).

\begin{proposition}\label{pro:2.17}
\cite{GZZ15a}. Let $(E, \mathcal{P})$ be an $RLC$ module over $K$ with base $(\Omega, \mathcal{F}, \mu)$, $x \in E$ and $M \subset E$ a nonempty $\mathcal{T}_{\varepsilon,\lambda}$--closed $L^0$--convex subset such that $x \notin M$. Then there exists $f \in E^{\ast}_{\varepsilon,\lambda}$ such that the following two conditions are satisfied:\\
$(1)$ $Ref(x) > \bigvee \{ Ref(y)\ |\ y \in M \}$ on $(d(x ,M) > 0)$;\\
$(2)$ $Ref(x) = \bigvee \{ Ref(y)\ |\ y \in M \}$ on $(d(x ,M) > 0)^c$;\\
where $Ref : E \to L^0(\mathcal{F})$ is defined by $Ref(z) = Re(f(z))$ $($namely the real part of $f(z))$ for all $z \in E$.

In addition, if $\tilde{I}_{A}\{x\} \bigcap \tilde{I}_{A}M = \emptyset$ for all $A \in \mathcal{F}$ with $\mu(A) > 0$, then (1) and (2) above can be simply stated as:\\
$(3)$ $Ref(x) > \bigvee \{ Ref(y)\ |\ y \in M \}$ on $\Omega$.
\end{proposition}

\begin{corollary}\label{cor:2.18}
\cite{GZZ15a}. Let $(E, \mathcal{P})$ be an $RLC$ module over $K$ with base $(\Omega, \mathcal{F}, \mu)$, $x \in E$ and $M \subset E$ a $\mathcal{T}_c$--closed $L^0$--convex nonempty subset such that $x \notin M$ and $M$ has the countable concatenation property. Then there exists $f \in E^{\ast}_c$ such that the following two conditions are satisfied:\\
$(1)$ $Ref(x) > \bigvee \{ Ref(y)\ |\ y \in M \}$ on $(d(x ,M) > 0)$;\\
$(2)$ $Ref(x) = \bigvee \{ Ref(y)\ |\ y \in M \}$ on $(d(x ,M) > 0)^c$.\\

In addition, if $\tilde{I}_{A}\{x\} \bigcap \tilde{I}_{A}M = \emptyset$ for all $A \in \mathcal{F}$ with $\mu(A) > 0$, then we have:\\
$(3)$ $Ref(x) > \bigvee \{ Ref(y)\ |\ y \in M \}$ on $\Omega$.
\end{corollary}

In the sequel, for any $\xi$ and $\eta$ in $\bar{L}^0(\mathcal{F})$, let $\xi^0$ and $\eta^0$ be arbitrarily chosen representatives of $\xi$ and $\eta$, respectively. We use $(\xi \geq \eta)$ for the set $\{ \omega \in \Omega\ |\ \xi^{0}(\omega) \geq \eta^{0}(\omega) \}$. Although $(\xi \geq \eta)$ depends on the choice of $\xi^0$ and $\eta^0$, $(\xi \geq \eta)$ only differs by a $\mu$--null set, this would not produce any confusion as long as we interpret the equality and inclusion relations between sets as the equality and inclusion almost everywhere. Similarly, one can understand $(\xi > \eta)$ and $(\xi = \eta)$.

Now, we can give the proof of Lemma \ref{lem:2.16}.

\noindent\textbf{Proof of Lemma \ref{lem:2.16}.} By Definition \ref{def:2.12} $epi(f)$ is a nonempty $\mathcal{T}_{\varepsilon,\lambda}$--closed $L^0$--convex subset in the random locally convex module $E \times L^0(\mathcal{F})$ whose family of $L^0$--seminorm is $\{ \|\cdot\| + |\cdot| : \|\cdot\| \in \mathcal{P} \}$, where for any $(x, r) \in E \times L^0(\mathcal{F})$, $(\|\cdot\| + |\cdot|)(x, r) = \|x\| + |r|$ for all $\|\cdot\| \in \mathcal{P}$. Since $f$ is local and $f(x_0) > \beta$ on $\Omega$, it is obvious that $\tilde{I}_{A}(x_0, \beta) \bigcap \tilde{I}_{A}epi(f) = \emptyset$ for all $A \in \mathcal{F}$ with $\mu(A) > 0$. By Proposition \ref{pro:2.17} there exists $(g_1, g_2) \in (E \times L^0(\mathcal{F}))^{\ast}_{\varepsilon,\lambda} = E^{\ast}_{\varepsilon,\lambda} \times (L^0(\mathcal{F}))^{\ast}_{\varepsilon,\lambda}$ (in fact, $(L^0(\mathcal{F}))^{\ast}_{\varepsilon,\lambda} = L^0(\mathcal{F})$) such that $g_1(x_0) + g_2(\beta) > \delta := \bigvee_{(x, y) \in epi(f)}(g_1(x) + g_2(y))$ on $\Omega$. This has the following consequences:\\
$(i)$ $g_2(1) \leq 0$.

Indeed, by noticing $g_2(y) = y g_2(1)$ for all $y \in L^0(\mathcal{F})$ and the fact that $(x, y)$ also belongs to $epi(f)$ whenever $(x, r) \in epif$ and $y \in L^0(\mathcal{F})$ satisfies $y \geq r$, then $g_1(x) + g_2(y)$ is also large enough on $(g_2(1) > 0)$ for larger $y \in L^0(\mathcal{F})$, which means $\mu(g_2(1) > 0) = 0$ since $g_1(x) + g_2(y)$ is bounded above by $g_1(x_0) + g_2(\beta)$.\\
$(ii)$ $(f(x_0) < +\infty) \subset (g_2(1) < 0)$.

Indeed, define $\hat{x}_0 = \tilde{I}_{(f(x_0) < +\infty)} x_0 + \tilde{I}_{(f(x_0) = +\infty)} x$ for some $x \in dom(f)$, then by $L^0$--convexity of $f, \hat{x}_0 \in dom(f)$. Further, local property of $f$ and the definition of $\delta$ imply $g_1(x_0) + g_2(f(x_0)) = g_1(\hat{x}_0) + g_2(f(\hat{x}_0)) < g_1(x_0) + g_2(\beta)$ on $(f(x_0) < +\infty)$. Hence $f(x_0)g_2(1) = g_2(f(x_0)) < g_2(\beta) = \beta g_2(1)$ on $(f(x_0) < +\infty)$, which implies that $g_2(1) < 0$ on $(f(x_0) < +\infty)$.

We distinguish the two cases $x_0 \in dom(f)$ and $x_0 \notin dom(f)$.

Case 1. Assume $x_0 \in dom(f)$. Then $g_2(1) < 0$ on $\Omega$ by $(ii)$. Thus, define $h$ by
\begin{equation*}
  h(x) = -\frac{g_1(x - x_0)}{g_2(1)} + \beta
\end{equation*}
for all $x \in E$, which satisfies our requirement.
Indeed, $h(x) \leq f(x)$ for all $x \in dom(f)$ by the definition of $\delta$. If $x \notin dom(f)$, we take some $x^{\prime \prime} \in dom(f)$ and $x^{\prime} = \tilde{I}_{B}x + \tilde{I}_{B^c}x^{\prime \prime}$, where $B = (f(x) < +\infty)$, then $\tilde{I}_{B}h(x) = \tilde{I}_{B}h(x^{\prime}) \leq \tilde{I}_{B}f(x)$ by noticing $x^{\prime} \in dom(f)$. Hence, $h(x) \leq f(x)$ for all $x \in E$ and it is obvious that $h(x_0) = \beta$.

Case 2. Assume $x_0 \notin dom(f)$. Then choose any $x^{\prime}_0 \in dom(f)$ and let $\beta^{\prime} = f(x^{\prime}_0) - \varepsilon$ for some $\varepsilon \in L^0_{++}(\mathcal{F})$, by Case 1 above there corresponds a $\mathcal{T}_{\varepsilon,\lambda}$--continuous affine function $h^{\prime} : E \to L^0(\mathcal{F})$ such that $h^{\prime}(x^{\prime}_0) = \beta^{\prime}$ and $h^{\prime}(x) \leq f(x)$ for all $x \in E$. Now, define $A_1
= (g_2(1) < 0), A_2 = A^c_1$ and $h_1, h_2 : E \to L^0(\mathcal{F})$ as follow:

\begin{equation*}
  h_1(x) = \tilde{I}_{A_1}\left(- \frac{g_1(x - x_0)}{g_2(1)} + \beta\right)
\end{equation*}
for all $x \in E$;
\begin{equation*}
  h_2(x) = \tilde{I}_{A_2}\left[h^{\prime}(x) + \tilde{I}_{(h^{\prime}(x_0) \geq \beta)}(\beta - h^{\prime}(x_0)) + \tilde{I}_{(h^{\prime}(x_0) < \beta)} \frac{\beta - h^{\prime}(x_0)}{\tilde{h}(x_0)}\tilde{h}(x)\right]
\end{equation*}
for all $x \in E$;\\
where $\tilde{h} : E \to L^0(\mathcal{F})$ is defined by $\tilde{h}(x) = \delta - g_1(x)$ for all $x \in E$, and we adopt the convention $\frac{0}{0} = 0$.

Finally, note that $\tilde{h}(x_0) < 0$ on $(g_2(1) = 0)$ and $\tilde{h}(x) \geq 0$ on $(g_2(1) = 0)$ for all $x \in dom(f)$. It follows that $h = h_1 + h_2$ is as required.

This completes the proof.\qed

\begin{remark}\label{rem:2.19}
Let $(E, \mathcal{P})$ be an $RLC$ module over $R$ with base $(\Omega, \mathcal{F}, \mu)$, $f : E \to \bar{L}^0(\mathcal{F})$ a $\mathcal{T}_c$--lower semicontinuous $L^0$--convex function and $\varepsilon \in L^0_{++}(\mathcal{F})$. $g \in E^{\ast}_c$ is called an $\varepsilon$--subgradient of $f$ at $x_0 \in dom(f)$, if $g(x - x_0) \leq f(x) - f(x_0) + \varepsilon$ for all $x \in E$. If $E$ and $\mathcal{P}$ both have the countable concatenation property, then for any $x_0 \in dom(f)$ and $\varepsilon \in L^0_{++}(\mathcal{F})$, $f$ has an $\varepsilon$--subgradient at $x_0$. In fact, by Lemma \ref{lem:2.16} and Theorem \ref{the:2.13} there exists $h = g + \alpha$ such that $h(x_0) = g(x_0) + \alpha = f(x_0) - \varepsilon$ and $h(x) \leq f(x)$ for all $x \in E$, where $g \in E^{\ast}_{\varepsilon,\lambda}$ and $\alpha \in L^0(\mathcal{F})$, and thus $\alpha = f(x_0) - g(x_0) - \varepsilon$ and $g(x) + f(x_0) - g(x_0) - \varepsilon \leq f(x)$ for all $x \in E$, namely $g(x - x_0) \leq f(x) - f(x_0) + \varepsilon$. Since $\mathcal{P}$ has the countable concatenation property, $E^{\ast}_c = E^{\ast}_{\varepsilon,\lambda}$, then $g$ also belongs to $E^{\ast}_c$, which is just an $\varepsilon$--subgradient of $f$ at $x_0$. From now on, we always denote by $\partial_{\varepsilon}f(x_0)$ the set of $\varepsilon$--subgradients of $f$ at $x_0$.
\end{remark}

As a corollary of Proposition \ref{pro:2.15}, we can get the following:
\begin{proposition}\label{pro:2.20}
Let $(E, \mathcal{P})$ be an $RLC$ module over $R$ with base $(\Omega, \mathcal{F}, \mu)$ such that $E$ has the countable concatenation property, and $f : E \to \bar{L}^0(\mathcal{F})$ a proper $\mathcal{T}_c$--lower semicontinuous $L^0$--convex function. Then $f^{\ast \ast}_c = f$.
\end{proposition}

\begin{proof}
We first consider the $RLC$ module $(E, \mathcal{P}_{cc})$ and let $\mathcal{T}^{\prime}_{\varepsilon,\lambda}$ and $\mathcal{T}^{\prime}_c$ be the $(\varepsilon, \lambda)$--topology and the locally $L^0$--convex topology on $E$ induced by $\mathcal{P}_{cc}$, respectively. It is obvious that $\mathcal{T}^{\prime}_c$ is stronger than $\mathcal{T}_c$, so $f$ is also $\mathcal{T}^{\prime}_c$--lower semicontinuous. Since $E$ and $\mathcal{P}_{cc}$ both have the countable concatenation property, $f$ is also $\mathcal{T}^{\prime}_{\varepsilon,\lambda}$--lower semicontinuous by Theorem \ref{the:2.13}, and hence also $\mathcal{T}_{\varepsilon,\lambda}$--lower semicontinuous since $\mathcal{P}_{cc}$ and $\mathcal{P}$ induce the same $(\varepsilon, \lambda)$--topology. By Proposition \ref{pro:2.15} $f = f^{\ast \ast}_{\varepsilon,\lambda}$, namely $f(x) = \bigvee \{ g(x) - f^{\ast}_{\varepsilon,\lambda}(g)\ |\ g \in E^{\ast}_{\varepsilon,\lambda} \}$. By Proposition \ref{pro:2.6} $E^{\ast}_{\varepsilon,\lambda} = H_{cc}(E^{\ast}_c)$, namely for each $g \in E^{\ast}_{\varepsilon,\lambda}$ there exists a sequence $\{ g_n\ |\ n \in N \}$ in $E^{\ast}_c$ and a countable partition $\{ A_n\ |\ n \in N \}$ of $\Omega$ to $\mathcal{F}$ such that $g = \sum^{\infty}_{n=1}\tilde{I}_{A_n}g_n$. Since for a fixed $x \in E$, $g(x) - f^{\ast}_{\varepsilon,\lambda}(g)$ is local with respect to $g$, one can have that $f(x) = \bigvee \{ g(x) - f^{\ast}_{\varepsilon,\lambda}(g)\ |\ g \in E^{\ast}_{\varepsilon,\lambda} \} = \bigvee \{ g(x) - f^{\ast}_{\varepsilon,\lambda}(g)\ |\ g \in H_{cc}(E^{\ast}_c) \} = \bigvee \{ g(x) - f^{\ast}_{\varepsilon,\lambda}(g)\ |\ g \in E^{\ast}_c \}$ by Lemma 5.2 of \cite{GZZ15a}. Again, by noticing $f^{\ast}_{\varepsilon,\lambda}|_{E^{\ast}_c} = f^{\ast}_c$, one can have that $f(x) = \bigvee \{ g(x) - f^{\ast}_c(g)\ |\ g \in E^{\ast}_c \} = f^{\ast \ast}_c(x)$.

This completes the proof.
\end{proof}

\begin{remark}\label{rem:2.21}
When $f$ satisfies the condition that $epi(f)$ is closed in $(E, \mathcal{T}_c) \times (L^0(\mathcal{F}), \mathcal{T}_c)$, Proposition \ref{pro:2.20} is exactly Theorem 5.2 of \cite{GZZ15a}, in fact, Proposition \ref{pro:2.20} has the same idea of proof as Theorem 5.2 of \cite{GZZ15a}.
\end{remark}

Let us conclude this section with some discussions on nonproper closed functions.

Let $(E, \mathcal{P})$ be an $RLC$ module over $R$ with base $(\Omega, \mathcal{F}, \mu)$ and $f : E \to \bar{L}^0(\mathcal{F})$ a local function. Let us recall some notation from \cite{GZZ15a} as follows:

$\mathcal{A} = \{ A \in \mathcal{F}\ |\ $there exists $x \in E$ such that $\tilde{I}_{A}f(x) = \tilde{I}_{A}(-\infty) \}$;

$\mathcal{B} = \{ A \in \mathcal{F}\ |\ \tilde{I}_{A}f(x) = \tilde{I}_{A}(+\infty)$ for all $x \in E \}$;

$MI(f) = esssup(\mathcal{A})$;

$PI(f) = esssup(\mathcal{B})$;

$BP(f) = \Omega \backslash (MI(f) \bigcup PI(f))$.

It is easy to check that $\tilde{I}_{PI(f)}f(x) = \tilde{I}_{PI(f)}(+\infty)$ for all $x \in E$ and $f(x) > - \infty$ on $BP(f)$ for all $x \in E$.

For each $D \in \mathcal{F}$ with $\mu(D) > 0$, let $E_D = \tilde{I}_{D}E = \{ \tilde{I}_{D}x\ |\ x \in E \}$ and $\mathcal{P}_D = \{ \|\cdot\|_D\ |\ \|\cdot\| \in \mathcal{P} \}$, where $\|\cdot\|_D$ stands for the restriction of $\|\cdot\|$ to $E_D$. Then $(E_D, \mathcal{P}_D)$ is an $RLC$ module over $R$ with base $(D, D \bigcap \mathcal{F}, \mu_D)$, where $\mu_D$ is the restriction of $\mu$ to $D \bigcap \mathcal{F}$. Further $f_D : E_D \to \tilde{I}_{D}L^0(\mathcal{F})$ is defined by $f_{D}(\tilde{I}_{D}x) = \tilde{I}_{D}f(x)$ for all $x \in E$, where $\tilde{I}_{D}L^0(\mathcal{F})$ is identified with $L^0(D \bigcap \mathcal{F})$.

\begin{definition}\label{def:2.22}
$f$ is said to be $\mathcal{T}_{\varepsilon,\lambda}$--(or, $\mathcal{T}_c$--)closed if $\tilde{I}_{MI(f)}f(x) = \tilde{I}_{MI(f)}(-\infty)$ for all $x \in E$ and $f_A$ is a proper $L^0(A \bigcap \mathcal{F})$--convex $\mathcal{T}_{\varepsilon,\lambda}$--(or, $\mathcal{T}_c$--)lower semicontinuous function on $(E_A, \mathcal{P}_A)$ for all $A \in \mathcal{F}$ with $A \subset BP(f)$ and $\mu(A) > 0$.
\end{definition}

Similar to the proof of Proposition 5.2 of \cite{GZZ15a}, one can have the following:

\begin{proposition}\label{pro:2.23}
Let $\{ f_{\alpha}, \alpha \in \Gamma \}$ be a family of $\mathcal{T}_{\varepsilon,\lambda}$--(respectively, $\mathcal{T}_c$--)closed functions from $(E, \mathcal{P})$ to $\bar{L}^0(\mathcal{F})$ and $f = \bigvee_{\alpha \in \Gamma}f_\alpha$ defined by $f(x) = \bigvee \{ f_{\alpha}(x)\ |\ \alpha \in \Gamma \}$ for all $x \in E$. Then $f$ is still $\mathcal{T}_{\varepsilon,\lambda}$--closed(respectively, $\mathcal{T}_c$--closed).
\end{proposition}

\begin{definition}\label{def:2.24}
Let $H = \{ g : E \to \bar{L}^0(\mathcal{F})\ |\ g \leq f$ and $g$ is $\mathcal{T}_{\varepsilon,\lambda}$--closed \}, then $cl_{\varepsilon,\lambda}(f) := \bigvee H$ is called the $\mathcal{T}_{\varepsilon,\lambda}$--closure of $f$. Similarly, one can also have the notion of $\mathcal{T}_c$--closure of $f$ (denoted by $cl_{c}(f)$).
\end{definition}

Theorem 5.3 of \cite{GZZ15a} shows $f^{\ast \ast}_{\varepsilon,\lambda} = cl_{\varepsilon,\lambda}(f)$ for all local function $f$ from an $RLC$ module $(E, \mathcal{P})$ to $\bar{L}^0(\mathcal{F})$. Similar to Corollary 5.1 of \cite{GZZ15a}, we also have the following:

\begin{proposition}\label{pro:2.25}
Let $(E, \mathcal{P})$ be an $RLC$ module over $R$ with base $(\Omega, \mathcal{F}, \mu)$ such that $E$ has the countable concatenation property, and $f : E \to \bar{L}^0(\mathcal{F})$ a local function. Then $f^{\ast \ast}_c = cl_{c}(f)$.
\end{proposition}

\section{Continuity}\label{sec:3}
Let$(E, \mathcal{P})$ be an $RLC$ module. An $L^0$--balanced, $L^0$--absorbent, $L^0$--convex and $\mathcal{T}_c$--closed subset of $E$ is called an $L^0$--barrel. $(E, \mathcal{P})$ is said to be $L^0$--pre--barrelled if every $L^0$--barrel with the countable concatenation property is a $\mathcal{T}_c$--neighborhood of $\theta$. In \cite{GZZ15b} it is proved that for an $RLC$ module $(E, \mathcal{P})$ such that $E$ has the countable concatenation property, then $(E, \mathcal{P})$ is $L^0$--pre--barrelled iff $\mathcal{T}_c = \beta (E, E^{\ast}_c)$, where $\beta (E, E^{\ast}_c)$ is the strongest random admissible topology of $E$ with respect to the natural random duality pair $\langle E, E^{\ast}_c \rangle$, in particular a $\mathcal{T}_c$--complete random normed module $(E, \|\cdot\|)$ such that $E$ has the countable concatenation property is $L^0$--pre--barrelled. In addition, Guo, et.al also established the following continuity theorem:

\begin{theorem}\label{the:3.1}
\cite{GZZ15b}. Let $(E, \mathcal{P})$ be an $L^0$--pre--barrelled $RLC$ module over $R$ with base $(\Omega, \mathcal{F}, \mu)$ such that $E$ has the countable concatenation property, and $f : E \to \bar{L}^0(\mathcal{F})$ a proper $\mathcal{T}_c$--lower semicontinuous $L^0$--convex function. Then $f$ is $\mathcal{T}_c$--continuous on $int(dom(f))$, namely $f$ is continuous from $(int(dom(f)), \mathcal{T}_c)$ to $(L^0(\mathcal{F}), \mathcal{T}_c)$, where $int(dom(f))$ stands for the $\mathcal{T}_c$--interior of $dom(f)$.
\end{theorem}

In the sequel, $int(dom(f))$ always denotes the $\mathcal{T}_c$--interior of $dom(f)$ for a proper function $f$ from an $RLC$ module over $R$ with base $(\Omega, \mathcal{F}, \mu)$ to $\bar{L}^0(\mathcal{F})$.

\begin{remark}\label{rem:3.2}
Let $(E, \mathcal{P})$ be an $RLC$ module over $R$ with base $(\Omega, \mathcal{F}, \mu)$ and $f : E \to \bar{L}^0(\mathcal{F})$ a proper $L^0$--convex function. In \cite{FKV09}, Filipovi\'c, et.al proved that the following three statements are equivalent : (i). $f$ is bounded above by some $\xi \in L^0(\mathcal{F})$ on a $\mathcal{T}_c$--neighborhood of some point $x_0$; (ii). $f$ is $\mathcal{T}_c$--continuous at $x_0$; (iii). $int(dom(f))$ is nonempty and $f$ is $\mathcal{T}_c$--continuous on $int(dom(f))$.
\end{remark}

For the sake of convenience, let us first give the following:
\begin{definition}\label{def:3.3}
Let $(E, \mathcal{P})$ be an $RLC$ module over $R$ with base $(\Omega, \mathcal{F}, \mu)$ and $f : E \to \bar{L}^0(\mathcal{F})$ a proper function. A sequence $\{ x_n\ |\ n \in N \}$ in $E$ is said to be convergent to $x \in E$ almost everywhere if $\{ \|x_n - x\|\ |\ n \in N \}$ converges to 0 almost everywhere for each $\|\cdot\| \in \mathcal{P}$. $f$ is said to be almost everywhere sequently continuous at $x_0 \in dom(f)$, if $\{ f(x_n)\ |\ n \in N \}$ is almost everywhere convergent to $f(x_0)$ for every sequence $\{ x_n\ |\ n \in N \}$ almost everywhere convergent to $x_0$. In addition, if $int(dom(f)) \neq \emptyset$ and $x_0 \in int(dom(f))$, f is said to be $L^0$--locally Lipschitzian at $x_0$, if there exist some $\mathcal{T}_c$--neighborhood $U$ of $x_0$, some $\xi \in L^0_{++}(\mathcal{F})$ and some $Q \in \mathcal{P}(F)$ such that $U \subset int(dom(f))$ and $|f(x) - f(y)| \leq \xi \|x -y\|_{Q}$ for all $x, y \in U$.
\end{definition}

\begin{theorem}\label{the:3.4}
Let $(E, \mathcal{P})$ be an $RLC$ module over $R$ with base $(\Omega, \mathcal{F}, \mu)$ and $f : E \to \bar{L}^0(\mathcal{F})$ a proper  $L^0$--convex function such that $f$ is $\mathcal{T}_c$--continuous at some point $x_0 \in int(dom(f))$. Then $f$ is $L^0$--locally Lipschitzian at $x_0$.
\end{theorem}

\begin{proof}
Let us first recall that $\{ U(Q, \varepsilon)\ |\ Q \in \mathcal{P}(F)$ and $\varepsilon \in L^0_{++}(\mathcal{F}) \}$ form a local base at $\theta$ of $\mathcal{T}_c$, where $U(Q, \varepsilon) = \{ x \in E\ |\ \|x\|_{Q} \leq \varepsilon \}$, then $\{ x_0 + U(Q, \varepsilon)\ |\ Q \in \mathcal{P}(F)$ and $\varepsilon \in L^0_{++}(\mathcal{F}) \}$ forms a local base at $x_0$ of $\mathcal{T}_c$. Since $f$ is $\mathcal{T}_c$--continuous at $x_0 \in int(dom(f))$, there exist some $Q \in \mathcal{P}(F)$ and $\delta \in L^0_{++}(\mathcal{F})$ such that $|f(x) - f(y)| \leq 1$ whenever $x$ and $y \in x_0 + U(Q, 2 \delta)$.

Denote $V = x_0 + U(Q, \delta)$. For any $y$ and $z$ in $V$, let $\alpha = \|y - z\|_Q$ and $w_n = y + \frac{\delta}{\alpha + \frac{1}{n}}(y - z)$ for all $n \in N$. Then $\|w_n - y\|_Q = \frac{\delta}{\alpha + \frac{1}{n}}\|y - z\|_Q \leq \delta$, so $\|w_n - x_0\|_Q \leq \|w_n - y\|_Q + \|y - x_0\|_Q \leq 2\delta$, namely $w_n \in x_0 + U(Q, 2 \delta)$.

Further, since $y = \frac{\alpha + \frac{1}{n}}{\alpha + \frac{1}{n} + \delta}w_n + \frac{\delta}{\alpha + \frac{1}{n} + \delta}z$, $f(y) \leq \frac{\alpha + \frac{1}{n}}{\alpha + \frac{1}{n} + \delta}f(w_n) + \frac{\delta}{\alpha + \frac{1}{n} + \delta}f(z)$, which implies that $f(y) - f(z) \leq \frac{\alpha + \frac{1}{n}}{\alpha + \frac{1}{n} + \delta}[f(w_n) - f(z)] \leq \frac{\alpha + \frac{1}{n}}{\alpha + \frac{1}{n} + \delta} \leq \frac{\alpha + \frac{1}{n}}{\delta}$ for all $n \in N$. Letting $n \to \infty$ will yield that $f(y) - f(z) \leq \frac{\alpha}{\delta} = \frac{1}{\delta}\|y - z\|_Q$. Switching the roles of $y$ and $z$ allows us to conclude $|f(y) - f(z)| \leq \frac{1}{\delta}\|y - z\|_Q$.

This completes the proof.
\end{proof}

\begin{theorem}\label{the:3.5}
Let $(E, \mathcal{P})$ be an $RLC$ module over $R$ with base $(\Omega, \mathcal{F}, \mu)$ and $f : E \to \bar{L}^0(\mathcal{F})$ a proper $L^0$--convex function such that $f$ is $\mathcal{T}_c$--continuous at $x_0 \in int(dom(f))$. Then $f$ is almost everywhere sequently continuous at $x_0$.
\end{theorem}

\begin{proof}
Without loss of generality, we can assume that $(\Omega, \mathcal{F}, \mu)$ is a probability space. Let$\{ x_n\ |\ n \in N \}$ be a sequence almost everywhere convergent to $x_0$, we only need to prove that there exists a sequence $\{ \varepsilon_k\ |\ k \in N \}$ in $L^0_{++}(\mathcal{F})$ such that $\varepsilon_k \downarrow 0$ and $\mu(\bigcap^{\infty}_{k=1} \bigcup^{\infty}_{n=1} \bigcap_{l \geq n} \{ \omega \in \Omega\ |\ |f(x_l)(\omega) - f(x_0)(\omega)| \leq \varepsilon_{k}(\omega) \}) = 1$.

According to the proof of Theorem \ref{the:3.4}, there exist some $Q \in \mathcal{P}(F)$ and $\delta \in L^0_{++}(\mathcal{F})$ such that $|f(y) - f(z)| \leq \frac{1}{\delta}\|y - z\|_Q$ whenever $y ,z \in x_0 + U(Q, \delta)$, in particular $|f(y) - f(x_0)| \leq \frac{1}{\delta}\|y - x_0\|_Q$ for all $y$ such that $\|y - x_0\|_Q \leq \delta$. We will show that $(\|y - x_0\|_Q \leq \delta) \subset (|f(y) - f(x_0)| \leq \frac{1}{\delta}\|y - x_0\|_Q)$ for all $y \in E$ as follows.

In fact, let $A = (\|y - x_0\|_Q \leq \delta)$ and $\tilde{y} = \tilde{I}_{A}y + \tilde{I}_{A^c}x_0$, then $\| \tilde{y} - x_0\|_Q = \tilde{I}_{A}\|y - x_0\|_Q + \tilde{I}_{A^c}\|x_0 -x_0\|_Q = \tilde{I}_{A}\|y - x_0\|_Q \leq \delta$, so $|f(\tilde{y}) - f(x_0)| \leq \frac{1}{\delta}\| \tilde{y} - x_0\|$. Since $f$ is local, $f(\tilde{y}) = \tilde{I}_{A}f(y) + \tilde{I}_{A^c}f(x_0)$, then one can have that $\tilde{I}_{A}|f(y) - f(x_0)| \leq \frac{1}{\delta} \tilde{I}_{A}\|y - x_0\|_Q$, which implies that $A \subset (|f(y) - f(x_0)| \leq \frac{1}{\delta}\|y - x_0\|_Q)$.

Let $\delta_k = \delta \wedge \frac{1}{k}$ for all $k \in N$, then $\delta_k \in L^0_{++}(\mathcal{F})$, since $\mu(\bigcap^{\infty}_{k=1} \bigcup^{\infty}_{n=1} \bigcap_{l \geq n} \{ \|x_l(\omega) - x_0(\omega)\|_Q \leq \delta_k(\omega) \}) = 1$ and $(\|x_l - x_0\|_Q \leq \delta_k) \subset (|f(x_l) - f(x_0)| \leq \frac{1}{\delta}\|x_l - x_0\|_Q) \subset (|f(x_l) - f(x_0)| \leq \frac{1}{\delta} \cdot \delta_k)$ for all $k$ and $l \in N$, we have that $\mu(\bigcap^{\infty}_{k=1} \bigcup^{\infty}_{n=1} \bigcap_{l \geq n} \{ |f(x_l)\\(\omega) - f(x_0)(\omega)| \leq \varepsilon_{k}(\omega) \}) = 1$, where $\varepsilon_k = \frac{\delta_k}{\delta}$.

This completes the proof.
\end{proof}

\begin{theorem}\label{the:3.6}
Let $(E, \|\cdot\|)$ be a $\mathcal{T}_c$--complete $RN$ module over $R$ with base $(\Omega, \mathcal{F}, \mu)$ such that $E$ has the countable concatenation property and $f : E \to L^0(\mathcal{F})$ an $L^0$--convex function. Then the following are equivalent:\\
$(1)$ $f$ is $\mathcal{T}_c$--lower semicontinuous;\\
$(2)$ $f$ is $\mathcal{T}_c$--continuous;\\
$(3)$ $f$ is almost everywhere sequently continuous;\\
$(4)$ $f$ is almost everywhere sequently lower semicontinuous, namely $\underline{lim}_{n}f(x_n) \geq f(x)$ for any sequence
$\{ x_n\ |\ n \in N \}$ such that $\{ \|x_n -x\|\ |\ n \in N \}$ converges to 0 almost everywhere;\\
$(5)$ $f$ is $\mathcal{T}_{\varepsilon,\lambda}$--continuous;\\
$(6)$ $f$ is $\mathcal{T}_{\varepsilon,\lambda}$--lower semicontinuous.
\end{theorem}

\begin{proof}
We can, without loss of generality, assume that $\mu(\Omega) = 1$. Since $dom(f) = E$, (1) implies (2) by Theorem \ref{the:3.1}.

$(2)$ $\Rightarrow$ $(3)$ is by Theorem \ref{the:3.5}.

$(3)$ $\Rightarrow$ $(4)$ is clear.

$(3)$ $\Rightarrow$ $(5)$: let $\{ x_n\ |\ n \in N \}$ be a sequence convergent to $x$ with respect to $\mathcal{T}_{\varepsilon,\lambda}$, namely $\{ \|x_n - x\|\ |\ n \in N \}$ converges to 0 in probability $\mu$, we only need to prove that for any subsequence $\{ f(x_{n_k})\ |\ k \in N \}$ of $\{ f(x_n)\ |\ n \in N \}$ there exists a subsequence $\{ f(x_{n_{k_{l}}})\ |\ l \in N \}$ such that $\{ |f(x_{n_{k_{l}}}) - f(x)|\ |\ l \in N \}$ converges to 0 almost surely. In fact, since $\{ \|x_{n_{k}} - x\|\ |\ k \in N \}$ still converges to 0 in probability $\mu$, there exists a subsequence $\{x_{n_{k_{l}}}\ |\ l \in N \}$ such that $\{ \|x_{n_{k_{l}}} - x\|\ |\ l \in N \}$ converges to 0 almost surely, which means that $\{ |f(x_{n_{k_{l}}}) - f(x)|\ |\ l \in N \}$ converges to 0 almost surely by the almost surely sequent continuity of $f$.

It is obvious that either of (4) and (5) implies (6).

$(6)$ $\Rightarrow$ (1) is by Theorem \ref{the:2.13}.

This completes the proof.
\end{proof}

\begin{remark}\label{rem:3.7}
In Theorem \ref{the:3.6}, when $f$ is only a proper and local function, one can also have that (4) implies (6) (equivalently, (1)). Thus the almost everywhere sequently lower semicontinuity employed in \cite{CKV12,FKV12} is a stronger hypothesis.
\end{remark}

\section{Subdifferential calculus}\label{sec:4}
Let $(E, \mathcal{P})$ be an $RLC$ module over $R$ with base $(\Omega, \mathcal{F}, \mu)$ and $f : E \to \bar{L}^0(\mathcal{F})$ a proper function. Let $x_0 \in dom(f)$, $g \in E^{\ast}_c$ is called a subgradient
of $f$ at $x_0$ if $g(x - x_0) \leq f(x) - f(x_0)$ for all $x \in E$. Denote by $\partial f(x_0)$ the set of subgradients of $f$ at $x_0$, $\partial f(x_0)$ is called the subdifferential of $f$ at $x_0$. If $\partial f(x_0) \neq \emptyset$, $f$ is said to be subdifferentiable at $x_0$.

Guo, et.al established the following subdifferentiability theorem in \cite{GZZ15b}.
\begin{proposition}\label{pro:4.1}
\cite{GZZ15b}. Let $(E, \mathcal{P})$ be an $L^0$--pre--barrelled $RLC$ module over $R$ with base $(\Omega, \mathcal{F}, \mu)$ such that $E$ has the countable concatenation property and $f : E \to \bar{L}^0(\mathcal{F})$ a proper $\mathcal{T}_c$--lower semicontinuous $L^0$--convex function. Then $\partial f(x) \neq \emptyset$ for all $x \in int(dom(f))$.
\end{proposition}

The proof of Proposition \ref{pro:4.1} used Proposition \ref{pro:4.2} below, which is stated as follows since it is frequently used in this paper.

Proposition \ref{pro:4.2} below is also true for any Hausdorff locally $L^{0}-$convex module.

\begin{proposition} \cite{FKV09}. \label{pro:4.2}
Let $(E,\mathcal{P})$ be an $RLC$ module over $K$ with base $(\Omega,\mathcal{F},\mu)$, $M$ and $G$ two nonempty $L^{0}-$convex subsets of $E$ with $G$ $\mathcal{T}_{c}-$open. If ${\tilde{I}_{A}M}\bigcap{\tilde{I}_{A}G}=\emptyset$ for all $A\in\mathcal{F}$ such that $\mu(A)>0$, then there exists $f\in E_{c}^{*}$ such that $Ref(x)>Ref(y)$ on $\Omega$ for all $x\in M$ and $y\in G$.
\end{proposition}

Generally, $\partial F_{1}(u)+\partial F_{2}(u)\subset \partial (F_{1}+F_{2})(u)$ for all $u\in E$. Conversely, we have the following:

\begin{theorem}\label{the:4.3}
Let $(E,\mathcal{P})$ be an $RLC$ module over $R$ with base $(\Omega,\mathcal{F},\mu)$, $F_{1},F_{2}:E\rightarrow \bar{L}^0(\mathcal{F})$ two proper $\mathcal{T}_{c}-$lower semicontinuous $L^{0}-$convex functions and $\bar{u}$ a point in $dom(F_{1})\bigcap dom(F_{2})$ such that $F_{1}$ is $\mathcal{T}_{c}-$continuous at $\bar{u}$. Then $\partial (F_{1}+F_{2})(u)=\partial F_{1}(u)+\partial F_{2}(u)$ for all $u\in E$.
\end{theorem}

\begin{proof}
We only need to prove that $\partial (F_{1}+F_{2})(u)\subset \partial F_{1}(u)+\partial F_{2}(u)$, that is, each $u^{*}\in \partial (F_{1}+F_{2})(u)$ can be decomposed into $u_{1}^{*}+u_{2}^{*}$, with $u_{1}^{*}\in \partial F_{1}(u)$ and $u_{2}^{*}\in \partial F_{2}(u)$. Our hypothesis means that $F_{1}(u)$ and $F_{2}(u)$ belong to $L^{0}(\mathcal{F})$ and that for all $v\in E$,
\begin{equation}
F_{1}(v)+F_{2}(v)\geq u^{*}(v-u)+F_{1}(u)+F_{2}(u).\label{4.1}
\end{equation}

Consider the two $L^{0}-$convex sets in $E\times L^{0}(\mathcal{F})$: $$C_{1}=\{(v, a)|F_{1}(v)-u^{*}(v-u)-F_{1}(u)\leq a\};~~~C_{2}=\{(v, a)|a\leq F_{2}(u)-F_{2}(v)\}.$$

Since $C_{1}$ is the epigraph of the function $G$ defined by $G(v)=F_{1}(v)-u^{*}(v)-F_{1}(u)+u^{*}(u)$ for all $v\in E$, which is proper, $L^{0}-$convex and $\mathcal{T}_{c}-$continuous at $\bar{u}$, it is easy to check that $C_{1}$ is an $L^{0}-$convex set with nonempty $\mathcal{T}_{c}-$interior. The inequality (\ref{4.1}) yields that $$\tilde{I}_{A} int(C_{1})\bigcap \tilde{I}_{A} C_{2}=\emptyset$$ for all $A\in \mathcal{F}$ with $\mu(A)>0$.

By Proposition \ref{pro:4.2}, there exists $(g_{1},g_{2})\in (E\times L^{0}(\mathcal{F}))_{c}^{*}=E_{c}^{*}\times L^{0}(\mathcal{F})_{c}^{*}$ (notice, $L^{0}(\mathcal{F})_{c}^{*}=L^{0}(\mathcal{F}))$ such that $g_{1}(v_{2})+g_{2}(a_{2})>g_{1}(v_{1})+g_{2}(a_{1})$ on $\Omega$ for all $(v_{2},a_{2})\in C_{2}$ and $(v_{1},a_{1})\in int(C_{1})$.

Denote $g_{2}(1)=\beta$, then the inequality above becomes $g_{1}(v_{2})+\beta a_{2}>g_{1}(v_{1})+\beta a_{1}$ on $\Omega$ for all $(v_{2}, a_{2})\in C_{2}$ and all $(v_{1}, a_{1})\in int(C_{1})$. Since $a_{1}$ may be arbitrarily large, $\beta$ must be strictly smaller than $0$ on $\Omega$ (namely $\beta<0$ on $\Omega$), and hence we have that $v^{*}(v_{2})+a_{2}<v^{*}(v_{1})+a_{1}$ on $\Omega$ (by denoting $v^{*}=\frac{g_{1}}{\beta}$) for all $(v_{2}, a_{2})\in C_{2}$ and all $(v_{1}, a_{1})\in int(C_{1})$, which implies that $v^{*}(v_{2})+a_{2}\leq v^{*}(v_{1})+a_{1}$ for all $(v_{2}, a_{2})\in C_{2}$ and all $(v_{1}, a_{1})\in C_{1}$ since $int(C_{1})$ is $\mathcal{T}_{c}-$dense in $C_{1}$.

By taking $v_{2}=v$, $a_{2}=F_{2}(u)-F_{2}(v)$, $v_{1}=u$ and $a_{1}=0$, one can get that $v^{*}(v-u)\leq F_{2}(v)-F_{2}(u)$ for all $v\in E$, namely $v^{*}\in \partial F_{2}(u)$. By taking $v_{2}=u$, $a_{2}=0$, $v_{1}=v$ and $a_{1}=F_{1}(v)-F_{1}(u)-u^{*}(v-u)$, one can have that $v^{*}(v)+F_{1}(v)-F_{1}(u)-u^{*}(v-u)\geq v^{*}(u)$, namely $(u^{*}-v^{*})(v-u)\leq F_{1}(v)-F_{1}(u)$ for all $v\in E$, and hence also $u^{*}-v^{*}\in \partial F_{1}(u)$. Let $u_{1}^{*}=u^{*}-v^{*}$ and $u_{2}^{*}=v^{*}$, then $u^{*}=u_{1}^{*}+u_{2}^{*}$.

This completes the proof.
\end{proof}

Let $(E_{1},\mathcal{P}_{1})$ and $(E_{2},\mathcal{P}_{2})$ be two $RLC$ modules over $K$ with base $(\Omega,\mathcal{F},\mu)$ and $\wedge : E_{1}\rightarrow E_{2}$ a continuous module homomorphism from $(E_{1},\mathcal{T}_{c})$ to $(E_{2},\mathcal{T}_{c})$. Define $\wedge^{*} : (E_{2})_{c}^{*}\rightarrow (E_{1})_{c}^{*}$ by $\wedge^{*}(g_{2})(x_{1})=g_{2}(\wedge (x_{1}))$ for all $g_{2}\in (E_{2})_{c}^{*}$ and $x_{1}\in E_{1}$, called the adjoint of $\wedge$.

Similar to the proof of Proposition $5.7$ of \cite{ET99}, one can make use of Proposition \ref{pro:4.2} to complete the proof of Theorem \ref{the:4.4} below.

\begin{theorem}\label{the:4.4}
Let $(E_{1},\mathcal{P}_{1})$ and $(E_{2},\mathcal{P}_{2})$ be two $RLC$ modules over $R$ with base $(\Omega,\mathcal{F},\mu)$, $\wedge : E_{1}\rightarrow E_{2}$ a continuous module homomorphism from $(E_{1},\mathcal{T}_{c})$ to $(E_{2},\mathcal{T}_{c})$ and $F: E_{2}\rightarrow \bar{L}^{0}(\mathcal{F})$ a proper $\mathcal{T}_{c}-$lower semicontinuous $L^{0}-$convex function such that $F$ is $\mathcal{T}_{c}-$continuous at some point $\wedge(\bar{u})$ (where $\bar{u}\in E_{1}$). Then $\partial(F\circ\wedge)(u)=\wedge^{*}\partial F(\wedge u)$ for all $u\in E_{1}$.
\end{theorem}

\section{G\^{a}teaux - and Fr\'{e}ch\'{e}t - differentiability}
\label{sec:5}

The main results of the section are Theorems \ref{the:5.7} and \ref{the:5.10} below.

For a net $(\xi_{\delta}, \delta\in \Gamma)$ in $L_{++}^{0}(\mathcal{F})$, $(\xi_{\delta}, \delta\in \Gamma)$ is said to be decreasing if $\xi_{\delta_{1}}\leq \xi_{\delta_{2}}$ for all $\delta_{1}, \delta_{2}\in \Gamma$ such that $\delta_{2}<\delta_{1}$, where $<$ is the partial order on $\Gamma$. We say that $\xi_{\delta}\downarrow 0$ if $(\xi_{\delta}, \delta\in \Gamma)$ is decreasing and $\bigwedge_{\delta\in \Gamma}\xi_{\delta}=0$.

\begin{lemma}\label{lem:5.1}
Let $(E,\mathcal{P})$ be an $RLC$ module over $R$ with base $(\Omega,\mathcal{F},\mu)$, $f: E\rightarrow \bar{L}^0(\mathcal{F})$ an $L^{0}-$convex function and $x_{0}\in E$. Then for any net $(\xi_{\delta}, \delta\in \Gamma)$ in $L_{++}^{0}(\mathcal{F})$ such that $\xi_{\delta}\downarrow 0$, we have that
$$\bigwedge\left\{\frac{f(x_{0}+\xi y)-f(x_{0})}{\xi}\left|\right.\xi\in L_{++}^{0}(\mathcal{F})\right\}=\bigwedge\left\{\frac{f(x_{0}+\xi_{\delta} y)-f(x_{0})}{\xi_{\delta}}\left|\right.\delta\in \Gamma\right\}$$ for all $y\in E$.
\end{lemma}

\begin{proof}
Denote $\eta_{1}=\bigwedge\left\{\frac{f(x_{0}+\xi y)-f(x_{0})}{\xi}|\xi\in L_{++}^{0}(\mathcal{F})\right\}$ and $\eta_{2}=\bigwedge\left\{\frac{f(x_{0}+\xi_{\delta} y)-f(x_{0})}{\xi_{\delta}}|\right.$\\$\left.\delta\in \Gamma\right\}$, it is clear that $\eta_{1}\leq \eta_{2}$. We only need to prove $\eta_{2}\leq \eta_{1}$ as follows.

Since $f$ is $L^{0}-$convex, it is easy to verify that $\frac{f(x_{0}+\xi y)-f(x_{0})}{\xi}\leq \frac{f(x_{0}+\eta y)-f(x_{0})}{\eta}$ for all $\xi$ and $\eta\in L_{++}^{0}(\mathcal{F})$ such that $\xi\leq \eta$, which also shows that the net $\left(\frac{f(x_{0}+\xi_{\delta} y)-f(x_{0})}{\xi_{\delta}}, \right.$\\$\left.\delta\in \Gamma\right)$ is decreasing. Thus there exists a decreasing sequence $\{\xi_{\delta_n}|n\in N\}$ such that $\frac{f(x_{0}+\xi_{\delta_{n}} y)-f(x_{0})}{\xi_{\delta_{n}}}\downarrow\eta_{2}$. Since $\xi_{\delta}\downarrow 0$, there exists a decreasing $\{\xi_{\delta_{n}^{'}}|n\in N\}\downarrow 0$, let $\bar{\delta}_{n}\in \Gamma$ be such that $\bar{\delta}_{n}\geq \delta_{n}$ and $\bar{\delta}_{n}\geq \delta^{\prime}_{n}$, then $\left\{\frac{f(x_{0}+\xi_{\bar{\delta}_{n}} y)-f(x_{0})}{\xi_{\bar{\delta}_{n}}}|n\in N\right\}$ converges to $\eta_{2}$ almost everywhere.

We can also assume, without loss of generality, that $\{\bar{\delta}_{n}|n\in N\}$ is increasing. For any given $\xi\in L_{++}^{0}(\mathcal{F})$, let $A_{n}=(\xi_{\bar{\delta}_{n}}\leq \xi)$, then $A_{n}\uparrow \Omega$. Since $f$ is also local, one can see that $\frac{f(x_{0}+\xi_{\bar{\delta}_{n}} y)-f(x_{0})}{\xi_{\bar{\delta}_{n}}}\leq \frac{f(x_{0}+\xi y)-f(x_{0})}{\xi}$ on $A_{n}$ for all $n\in N$, letting $n\rightarrow +\infty$ one can have that $\eta_{2}\leq \frac{f(x_{0}+\xi y)-f(x_{0})}{\xi}$, which further shows that $\eta_{2}\leq \eta_{1}$.

This completes the proof.
\end{proof}

\begin{definition}\label{def:5.2}
Let $(E,\mathcal{P})$ be an $RLC$ module over $R$ with base $(\Omega,\mathcal{F},\mu)$ and $f: E\rightarrow \bar{L}^0(\mathcal{F})$ be a proper and local function. Further, let $x_{0}\in dom(f)$.\\
$(1)$ $f$ is said to have the directional derivative at $x_{0}$ if \\$f^{'}(x_{0},y):=\lim_{t_{n}\downarrow 0}\frac{f(x_{0}+t_{n}y)-f(x_{0})}{t_{n}}$ exists almost everywhere for all $y\in E$ and all decreasing sequence $\{t_{n}|n\in N\}$ of positive numbers such that $t_{n}\downarrow 0$, where $f^{'}(x_{0},y)$ is allowed to take values in $\bar{L}^{0}(\mathcal{F})$, called the directional derivative of $f$ at $x_{0}$ along the direction $y$.\\
$(2)$ $f$ is said to be G\^{a}teaux-differentiable at $x_{0}$ if $f$ has the directional derivative at $x_{0}$ and there exists $u\in E_{c}^{*}$ such that $f^{'}(x_{0},y)=u(y)$ for all $y\in E$, in which case $u$ is called the G\^{a}teaux-derivative of $f$ at $x_{0}$, denoted by $f^{'}(x_{0})$.\\
$(3)$ If $(E,\mathcal{P})$ is an $RN$ module (for example, $(E,\|\cdot\|)$) and there exists $u\in E_{c}^{*}$ such that $$\left\{\frac{f(x_{0}+h_{n})-f(x_{0})-u(h_{n})}{\|h_{n}\|}\left|\right.n\in N\right\}~\mbox{converges to 0 almost everywhere}$$ for all sequence $\{h_{n}|n\in N\}$ such that $\{\|h_{n}\||n\in N\}$ converges to 0 almost everywhere, where we adopt the convention $\frac{0}{0}=0$, in which case $f$ is said to be Fr\'{e}ch\'{e}t-differentiable at $x_{0}$ and $u$ is called the Fr\'{e}ch\'{e}t-derivative of $f$ at $x_{0}$, denoted by $\nabla f(x_{0})$.
\end{definition}

\begin{remark}\label{rem:5.3}
Lemma \ref{lem:5.1} shows that $f^{'}(x_{0},y)$ always exists and is equal to \\$\bigwedge\left\{\frac{f(x_{0}+\xi  y)-f(x_{0})}{\xi}|\xi\in L_{++}^{0}(\mathcal{F})\right\}$ when $f$ is an $L^{0}-$convex function, in which case
$$f^{'}(x_{0},y)=~\mbox{the limit almost everywhere of}~ \left\{\frac{f(x_{0}+\xi_{n} y)-f(x_{0})}{\xi_{n}}|n\in N\right\}$$ for any sequence $\{\xi_{n}|n\in N\}$ in $L_{++}^{0}(\mathcal{F})$ such that $\xi_{n}\downarrow 0$. Clearly, when $f^{'}(x_{0},y)$ exists, $\lim_{t_{n}\uparrow 0}\frac{f(x_{0}+t_{n}y)-f(x_{0})}{t_{n}}$ also exists and is exactly $-f^{'}(x_{0}, -y)$ for any sequence $\{t_{n}|n\in N\}$ of negative numbers such that $t_{n}\uparrow 0$. Further, when $f$ is G\^{a}teaux-differentiable at $x_{0}$, $\lim_{t_{n}\rightarrow0}\frac{f(x_{0}+t_{n}y)-f(x_{0})}{t_{n}}$ exists for any sequence $\{t_{n}|n\in N\}$ of real numbers such that $t_{n}\rightarrow 0$. Finally, when $f$ is Fr\'{e}ch\'{e}t-differentiable $f$ is also G\^{a}teaux-differentiable and the two kinds of derivatives concide.
\end{remark}

To study the properties of $f^{'}(x_{0},y)$ for an $L^{0}-$convex function $f$, we give Proposition \ref{pro:5.4} below, whose proof is omitted since it is completely a copy of classical Lemma $5.41$ of \cite{AB06}.

\begin{proposition}\label{pro:5.4}
Let $(E,\mathcal{P})$ be an $RLC$ module over $R$ with base $(\Omega,\mathcal{F},\mu)$, $f: E\rightarrow \bar{L}^0(\mathcal{F})$ a proper $L^{0}-$convex function and $x_{0}\in dom(f)$. Then $$|f(x_{0}+\lambda y)-f(x_{0})|\leq \lambda \max \{f(x_{0}+y)-f(x_{0}), f(x_{0}-y)-f(x_{0})\}$$ for all $y\in E$ and $\lambda \in L_{+}^{0}(\mathcal{F})$ such that $0\leq\lambda\leq1$.
\end{proposition}

Let us recall that a function $h: E\rightarrow \bar{L}^0(\mathcal{F})$ is called an $L^{0}-$sublinear function on an $L^{0}(\mathcal{F})-$module $E$ if $f$ is subadditive, namely $f(x+y)\leq f(x)+f(y)$ for all $x,y\in E$ and $f$ is also $L^{0}-$positively homogeneous, namely $f(\xi x)=\xi f(x)$ for all $\xi\in L_{+}^{0}(\mathcal{F})$ and $x\in E$.

\begin{theorem}\label{the:5.5}
Let $(E,\mathcal{P})$ be an $RLC$ module over $R$ with base $(\Omega,\mathcal{F},\mu)$, $f: E\rightarrow \bar{L}^0(\mathcal{F})$ a proper $L^{0}-$convex function and $x_{0}\in dom(f)$. Then we have the following statements:\\
$(1)$ $f^{'}(x_{0}, \cdot): E\rightarrow \bar{L}^0(\mathcal{F})$ is an $L^{0}-$sublinear function;\\
$(2)$ If $f$ is $\mathcal{T}_{c}-$continuous at $x_{0}$, then $f^{'}(x_{0}, \cdot ): E\rightarrow \bar{L}^0(\mathcal{F})$ is also $\mathcal{T}_{c}-$continuous and $f^{'}(x_{0}, y)\in L^{0}(\mathcal{F})$ for all $y\in E$.
\end{theorem}

\begin{proof}
$(1)$ First, we prove that $f^{'}(x_{0}, \cdot): E\rightarrow \bar{L}^0(\mathcal{F})$ is $L^{0}-$convex: in fact,
\begin{equation*}
\begin{aligned}
f^{'}(x_{0}, \lambda y_{1}+(1-\lambda)y_{2})&=\lim_{t_{n}\downarrow0} \frac{f(x_{0}+t_{n}[\lambda y_{1}+(1-\lambda)y_{2}])-f(x_{0})}{t_{n}} \\
&\leq \lim_{t_{n}\downarrow0} \left\{\lambda\left[\frac{f(x_{0}+t_{n}y_{1})-f(x_{0})}{t_{n}}\right]+(1-\lambda)\right.\\
&~~~~~~~~~~~~~~~~~~~~\left.~~~~~~~~~~~~~~~~~~~~~~~~~~~~~~~~~~~~~~~\cdot\left[\frac{f(x_{0}+t_{n}y_{2})-f(x_{0})}{t_{n}}\right]\right\}\\
&=\lambda f{'}(x_{0}, y_{1})+(1-\lambda)f{'}(x_{0}, y_{2}),
\end{aligned}
\end{equation*}
for all $y_{1}, y_{2}\in E$ and $\lambda\in L_{+}^{0}(\mathcal{F})$ such that $0\leq \lambda \leq 1$. Then, we prove that $f^{'}(x_{0}, \cdot): E\rightarrow \bar{L}^0(\mathcal{F})$ is $L^{0}-$positively homogeneous: if $\xi\in L_{++}^{0}(\mathcal{F})$ and $y\in E$, $$f^{'}(x_{0}, \xi y)=\lim_{t_{n}\downarrow0} \frac{f(x_{0}+t_{n}\xi y)-f(x_{0})}{t_{n}}=\xi \lim_{t_{n}\xi\downarrow0} \frac{f(x_{0}+t_{n}\xi y)-f(x_{0})}{t_{n}\xi}=\xi f^{'}(x_{0}, y)$$ by Remark \ref{rem:5.3}; if $\xi\in L_{+}^{0}(\mathcal{F})$, let $A=(\xi>0)$ and $A^{c}=(\xi=0)$, then by the local property of $f$, one can see that $f^{'}(x_{0}, \xi y)=0=\xi f^{'}(x_{0}, y)$ on $A^{c}$, further let $E_{A}=\tilde{I_{A}}E$, $\mathcal{P}_{A}=\left\{\|\cdot\|_{E_{A}}\left|\right.\|\cdot\|\in \mathcal{P}\right\}$ and $f_{A}: E_{A}\rightarrow \tilde{I}_{A}\bar{L}^{0}(\mathcal{F})$ be defined by $f_{A}(\tilde{I}_{A}x)=\tilde{I}_{A}f(x)$, then $f_{A}$ is an $L^{0}(A\bigcap \mathcal{F})-$convex function on the $RLC$ module $(E_{A}, \mathcal{P}_{A})$ with base $(A,A\bigcap\mathcal{F},\mu_{A})$, by the case we have proved, one can see that $\tilde{I}_{A}f^{'}(x_{0}, \xi y)=f_{A}^{'}(\tilde{I}_{A}x_{0}, \tilde{I}_{A}\xi\cdot \tilde{I}_{A}y)=\tilde{I}_{A}\xi\cdot f_{A}^{'}(\tilde{I}_{A}x_{0}, \tilde{I}_{A}y)=\tilde{I}_{A}\xi\cdot f^{'}(x_{0}, y)$, namely $f^{'}(x_{0}, \xi y)=\xi f^{'}(x_{0}, y)$ on $A$. To sum up, we have that $f^{'}(x_{0}, \xi y)=\xi f^{'}(x_{0}, y)$.

$(2)$ Since $f$ is $\mathcal{T}_{c}-$continuous at $x_{0}$, there exists an $L^{0}-$absorbient, $L^{0}-$balanced and $L^{0}-$convex $\mathcal{T}_{c}-$neighborhood $\mathcal{U}$ of $\theta$ for any $\varepsilon\in L_{++}^{0}(\mathcal{F})$ such that $|f(x_{0}+u)-f(x_{0})|\leq \varepsilon$ for all $u\in \mathcal{U}$, then
\begin{equation*}
\begin{aligned}
\left|f^{'}(x_{0}, u)\right|&=\lim_{t_{n}\downarrow0} \left|\frac{f(x_{0}+t_{n}u)-f(x_{0})}{t_{n}}\right|\\
&\leq \max\left\{f(x_{0}+u)-f(x_{0}),f(x_{0}-u)-f(x_{0})\right\}\\
&\leq \varepsilon .
\end{aligned}
\end{equation*}
Thus $f^{'}(x_{0}, \cdot): E\rightarrow \bar{L}^0(\mathcal{F})$ is $\mathcal{T}_{c}-$continuous at $\theta$, which yields that $f^{'}(x_{0}, \cdot)$ is $\mathcal{T}_{c}-$continuous at any point $y\in E$ and $f^{'}(x_{0}, y)\in L^{0}(\mathcal{F})$.

This completes the proof.
\end{proof}

To prove Theorem \ref{the:5.7} below, we first give Lemma \ref{lem:5.6} below.

\begin{lemma}\label{lem:5.6}
Let $(E,\mathcal{P})$ be an $RLC$ module over $R$ with base $(\Omega,\mathcal{F},\mu)$, $f: E\rightarrow \bar{L}^0(\mathcal{F})$ a proper $L^{0}-$convex function such that $f$ is $\mathcal{T}_{c}-$continuous at $x_{0}\in E$. Then $$f(x_{0})+\lambda f^{'}(x_{0},y)\leq f(x_{0}+\lambda y)$$ for all $y\in E$ and $\lambda\in L^{0}(\mathcal{F})$.
\end{lemma}

\begin{proof}
We proceed in these cases.

Case $1$. Let $\lambda\in L_{++}^{0}(\mathcal{F})$. Since $f^{'}(x_{0}, y)\leq \frac{f(x_{0}+\lambda y)-f(x_{0})}{\lambda}$, $\lambda f^{'}(x_{0}, y)+f(x_{0})\leq f(x_{0}+\lambda y)$ for all $y\in E$.

Case $2$. Let $\lambda\in L^{0}(\mathcal{F})$ be such that $\lambda <0$ on $\Omega$. Since $f^{'}(x_{0}, -y)\leq $\\ $\frac{f(x_{0}+(-\lambda)(-y))-f(x_{0})}{-\lambda}$ and $0=f^{'}(x_{0}, y-y)\leq f^{'}(x_{0}, y)+f^{'}(x_{0}, -y)$, one has that $\lambda f^{'}(x_{0}, y)=(-\lambda)\cdot (-f^{'}(x_{0}, y))\leq -\lambda f^{'}(x_{0}, -y)\leq f(x_{0}+\lambda y)-f(x_{0}),$ so $f(x_{0})+\lambda f^{'}(x_{0}, y)\leq f(x_{0}+\lambda y)$ for all $y\in E$.

Case $3$. Let $\lambda\in L^{0}(\mathcal{F})$. Let $A=(\lambda >0)$, $B=(\lambda =0)$ and $C=(\lambda <0)$. First, one can see from the local property of $f$ that $f(x_{0})+\lambda f^{'}(x_{0}, y)\leq f(x_{0}+\lambda y)$ on $B$. Then, as in the proof of Theorem \ref{the:5.5}, consider the $RLC$ module $(E_{A},\mathcal{P}_{A})$ and the $L^{0}(A\bigcap \mathcal{F})-$convex function $f_{A}: E_{A}\rightarrow \tilde{I}_{A}\bar{L}^0(\mathcal{F})$, one can see that $f(x_{0})+\lambda^{+}f^{'}(x_{0}, y)\leq f(x_{0}+\lambda^{+}y)$ on $A$ by Case $1$ we have proved. Similarly, by considering the corresponding case on $C$, one can see that $f(x_{0})-\lambda^{-}f^{'}(x_{0}, y)\leq f(x_{0}-\lambda^{-}y)$ on $C$ by Case $2$ we have proved, and hence $f(x_{0})+\lambda f^{'}(x_{0}, y)=\tilde{I}_{A}(f(x_{0})+\lambda f^{'}(x_{0}, y))+\tilde{I}_{B}(f(x_{0})+\lambda f^{'}(x_{0}, y))+\tilde{I}_{C}(f(x_{0})+\lambda f^{'}(x_{0}, y))\leq f(x_{0}+\lambda y)$ by the local property of $f$, where $\lambda^{+}=\lambda \bigvee 0$ and $\lambda^{-}=(-\lambda) \bigvee 0$.

This completes the proof.
\end{proof}

\begin{theorem}\label{the:5.7}
Let $(E,\mathcal{P})$ be an $RLC$ module over $R$ with base $(\Omega,\mathcal{F},\mu)$ and $f: E\rightarrow \bar{L}^0(\mathcal{F})$ a proper $L^{0}-$convex function. If $f$ is G\^{a}teaux-differentiable at $x_{0}\in E$, then it is subdifferentiable at $x_{0}$ and $\partial f(x_{0})=\{f^{'}(x_{0})\}$. Conversely, if $f$ is $\mathcal{T}_{c}-$continuous at $x_{0}$ and has only one subgradient, then $f$ is G\^{a}teaux-differentiable at $x_{0}$ and $\partial f(x_{0})=\{f^{'}(x_{0})\}$.
\end{theorem}

\begin{proof}
If $f$ is G\^{a}teaux-differentiable at $x_{0}$. First, it is obvious that $f^{'}(x_{0})\in \partial f(x_{0})$: in fact, for all $y\in E$, since $f(y)-f(x_{0})\geq f^{'}(x_{0}, y-x_{0})=f^{'}(x_{0})(y-x_{0})$, then $f^{'}(x_{0})\in \partial f(x_{0})$. Then, let $g\in E_{c}^{*}$ be any element of $\partial f(x_{0})$, then for all $y\in E$ and any given sequence $\{t_{n}|n\in N\}$ of positive numbers such that $t_{n}\downarrow 0$, $f(x_{0}+t_{n}y)-f(x_{0})\geq t_{n} g(y)$ for all $n\in N$, so $$f^{'}(x_{0})(y)=\lim_{t_{n}\downarrow 0} \frac{f(x_{0}+t_{n}y)-f(x_{0})}{t_{n}}\geq g(y)$$ for all $y\in E$, which must implies that $f^{'}(x_{0})=g$. This completes the proof of the first part of the theorem.

Let us turn to the second part. Since $f$ is $\mathcal{T}_{c}-$continuous at $x_{0}$, then by $(2)$ of Theorem $\ref{the:5.5}$, $f^{'}(x_{0}, y)\in L^{0}(\mathcal{F})$ for all $y\in E$. Further, by Lemma $\ref{lem:5.6}$, $f(x_{0})+\lambda f^{'}(x_{0},y)\leq f(x_{0}+\lambda y)$ for all $y\in E$ and all $\lambda\in L^{0}(\mathcal{F})$. Let $L(y)=\{(x_{0}+\lambda y, f(x_{0})+\lambda f^{'}(x_{0},y))| \lambda\in L^{0}(\mathcal{F})\}$ for each $y\in E$. Since $int(epi(f))\neq \emptyset$ and $f$ is local, it is easy to verify, for each $y\in E$, that $L(y)$ is an $L^{0}-$convex set in $E\times L^{0}(\mathcal{F})$ and $\tilde{I}_{A}L(y)\bigcap\tilde{I}_{A} int(epi(f))=\emptyset$ for all $A\in \mathcal{F}$ with $\mu(A)>0$. As in the proof of Theorem \ref{the:4.3}, there exists $v_{y}^{*} \in E^{*}_c$ for each $y\in E$ such that $v_{y}^{*}(v_{2})+a_{2}\leq v_{y}^{*}(v_{1})+a_{1}$ for all $(v_{2},a_{2})\in L(y)$ and $(v_{1},a_{1})\in epi(f)$. Since $L(y)=(x_{0}, f(x_{0}))+\{\lambda(y, f^{'}(x_{0}, y))|\lambda\in L^{0}(\mathcal{F})\}$, $v_{y}^{*}$ is forced to satisfy that $v_{y}^{*}(y)+f^{'}(x_{0}, y)=0$. Further, $v_{y}^{*}(x_{0})+f(x_{0})\leq v_{y}^{*}(v)+f(v)$ for all $v\in dom(f)$, which means that $-v_{y}^{*}\in \partial f(x_{0})$. Since $\partial f(x_{0})$ is a singleton, for example, let $\partial f(x_{0})=\{u^{*}\}$ for some $u^{*}\in E_{c}^{*}$, then $-v_{y}^{*}=u^{*}$ for all $y\in E$, so that $u^{*}(y)=f^{'}(x_{0}, y)$ for all $y\in E$, namely, $f$ is G\^{a}teaux-differentiable at $x_{0}$ and $f^{'}(x_{0})=u^{*}$.

This completes the proof.
\end{proof}

Let $(E,\mathcal{P})$ be an $RLC$ module over $R$ with base $(\Omega,\mathcal{F},\mu)$ and $\mathcal{A}$ an $L^{0}-$convex subset of $E$. We say that a proper function $f: E\rightarrow \bar{L}^0(\mathcal{F})$ is:\\
$(1)$ $L^{0}-$convex on $\mathcal{A}$ if $f(\lambda x+(1-\lambda)y)\leq \lambda f(x)+(1-\lambda)y$ for all $x, y\in \mathcal{A}$ and $\lambda\in L_{+}^0(\mathcal{F})$ with $0\leq\lambda\leq1$;\\
$(2)$ strictly $L^{0}-$convex if $f(\lambda x+(1-\lambda)y)< \lambda f(x)+(1-\lambda)y$ on $(0<\lambda<1)\bigcap (d(x,y)>0)$ for all $x, y\in E$ and $\lambda\in L_{+}^0(\mathcal{F})$ with $0\leq\lambda\leq1$, where $d(x, y)=\bigvee \left\{\|x-y\|\left|\right.\|\cdot\|\in \mathcal{P}\right\}$. \\
Similar to the classical case (see, for instance, Propositions $5.4$ and $5.5$ of \cite{ET99}), one can have the following:

\begin{theorem}\label{the:5.8}
Let $(E,\mathcal{P})$ be an $RLC$ module over $R$ with base $(\Omega,\mathcal{F},\mu)$, $\mathcal{A}$ an $L^{0}-$convex subset of $E$ and $f: E\rightarrow \bar{L}^0(\mathcal{F})$ a proper function such that $f$ is G\^{a}teaux-differentiable on $\mathcal{A}$. Then we have the following statements:\\
$(1)$ $f$ is $L^{0}-$convex on $\mathcal{A}$ iff $f(y)\geq f(x)+f^{'}(x)(y-x)$ for all $x, y \in \mathcal{A}$;\\
$(2)$ If $f$ is $L^{0}-$convex on $\mathcal{A}$, then $f^{'}(\cdot): \mathcal{A}\rightarrow E_{c}^{*}$ is monotone, namely, $(f^{'}(x)-f^{'}(y))(x-y)\geq 0$ for all $x, y \in \mathcal{A}$;\\
$(3)$ $f$ is strictly $L^{0}-$convex on $\mathcal{A}$ iff $f(y)> f(x)+f^{'}(x)(y-x)$ on $(d(x, y)>0)$ for all $x, y \in \mathcal{A}$.
\end{theorem}

We will end the section with Theorem \ref{the:5.10} below which gives the relation between G\^{a}teaux-and Fr\'{e}ch\'{e}t-differentiability. For this, we need the following:

\begin{definition}\label{def:5.9}
Let $(E, \|\cdot\|)$ be an $RN$ module over $R$ with base $(\Omega,\mathcal{F},\mu)$ and $f: E\rightarrow \bar{L}^0(\mathcal{F})$ a proper and local function such that $f$ is G\^{a}teaux-differentiable at some $x_{0}\in dom(f)$. $f$ is said to be almost everywhere sequently continuously G\^{a}teaux-differentiable at $x_{0}$ if there exists some $\mathcal{T}_{c}-$neighborhood $V$ of $x_{0}$ such that $f$ is G\^{a}teaux-differentiable on $V$ and $\{f^{'}(x_{n})|n\in N\}$ converges to $f^{'}(x_{0})$ almost everywhere (namely $\{\|f^{'}(x_{n})-f^{'}(x_{0})\|\left|\right.n\in N\}$ converges to $0$ almost everywhere) whenever $\{x_{n}|n\in N\}$ is a sequence in $V$ such that $\left\{\|x_{n}-x_{0}\|\left|\right. n\in N\right\}$ converges to $0$ almost everywhere, where, for an element $g$ in $E_{c}^{*}$, the $L^{0}-$norm $\|g\|$ is defined by $\|g\|=\bigvee\left\{|g(x)|\left|\right.x\in E\right.$ and $\left.\|x\|\leq1\right\}$.
\end{definition}

Now, we can state Theorem \ref{the:5.10} as follows:

\begin{theorem}\label{the:5.10}
Let $(E, \|\cdot\|)$ be a $\mathcal{T}_{c}-$complete $RN$ module over $R$ with base $(\Omega,\mathcal{F},\mu)$ such that $E$ has the countable concatenation property and $f: E\rightarrow \bar{L}^0(\mathcal{F})$ a proper $L^{0}-$convex function such that $f$ is almost everywhere sequently continuously G\^{a}teaux-differentiable at some $x_{0}\in dom(f)$. Then $f$ is Fr\'{e}ch\'{e}t - differentiable at $x_{0}$.
\end{theorem}

The proof of Theorem \ref{the:5.10} needs some work on Riemann calculus for abstract functions from a finite real interval to an $RN$ module, which was first established by Guo and Zhang in \cite{GZ12}. Let us first recall: let $[a, b]$ be a closed finite real interval, $(E, \|\cdot\|)$ an $RN$ module over $K$ with base $(\Omega,\mathcal{F},\mu)$ and $f: [a, b]\rightarrow E$ an $E-$valued function on $[a, b]$. For any partition $\triangle: a=t_{0}<t_{1}<t_{2}<\cdots<t_{n-1}<t_{n}=b$, $\xi_{i}\in[t_{i-1}, t_{i}]$ for all $i$ such that $1\leq i \leq n$. Let $\|\triangle\|=\max_{1\leq i \leq n }(t_{i}-t_{i-1})$ and $\mathcal{R}(f, \triangle, \{\xi_{i}\}_{i=1}^{n})=\Sigma_{i=1}^{n} f(\xi_{i})(t_{i}-t_{i-1})$.

\begin{definition} \cite{GZ12}. \label{def:5.11}
Let $[a, b]$ be a closed finite real interval, $(E, \|\cdot\|)$ a $\mathcal{T}_{\varepsilon, \lambda}-$complete $RN$ module over $K$ with base $(\Omega,\mathcal{F},\mu)$, $f: [a, b]\rightarrow E$ an $E-$valued function defined on $[a, b]$ and $\hat{\mu}$ be the probability measure associated with $\mu$.\\
$(1)$ $f$ is said to be Riemann integrable if there exists $I\in E$ with the property: there exists a positive number $\delta$ for any given positive numbers $\varepsilon$ and $\lambda$ with $0<\lambda<1$ such that $$\hat{\mu}\left\{\omega\in\Omega\left|\right.\|\mathcal{R}(f, \triangle, \{\xi_{i}\}_{i=1}^{n})-I\|(\omega)<\varepsilon\right\}>1-\lambda$$ whenever $\|\triangle\|<\delta$ for an arbitrary partition $\triangle: a=t_{0}<t_{1}<\cdots<t_{n}=b$, and an arbitrary choice of $\{\xi_{i}|1\leq i \leq n\}$ such that $\xi\in [t_{i-1}, t_{i}]$ with $1\leq i \leq n$, at which time $I$ is called the Riemann integration of $f$ over $[a, b]$, denoted by ${\int_{a}^{b} f(t) dt}$. \\
$(2)$ $f$ is said to be Riemann derivable at some $t_{0}\in [a, b]$ if $\lim\limits_{n \to \infty}{\frac{f(t_{n})-f(t_{0})}{t_{n}-t_{0}}}$ exists (denoted by $f^{'}(t_{0})$) with respect to $\mathcal{T}_{\varepsilon, \lambda}$ for any sequence $\{t_{n}|n\in N\}$ in $[a, b]$ such that $t_{n}\rightarrow t_{0}$. Further, $f$ is said to be Riemann derivable on $[a, b]$ if it is Riemann derivable at every point $t\in [a, b]$.
\end{definition}

It is easy to see that a monotone function $f: [a, b]\rightarrow L^{0}(\mathcal{F})$ (namely, $f(t_{1})\leq f(t_{2})$ for any $t_{1}, t_{2}\in [a,b]$ such that $t_{1}\leq t_{2}$) is Riemann integrable. In \cite{GZ12}, Guo and Zhang provide another class of Riemann integrable functions: if $f$ is a continuous function from $[a, b]$ to a $\mathcal{T}_{\varepsilon, \lambda}-$complete $RN$ module $(E, \|\cdot\|)$ endowed with the $(\varepsilon, \lambda)-$topology and satisfies the condition that $\bigvee\left\{\|f(t)\|\left|\right.t\in [a, b]\right\}\in L_{+}^{0}(\mathcal{F})$, then $f$ is Riemann integrable. In particular, the following Newton-Leibniz formula was established in \cite{GZ12}.

\begin{lemma} \cite{GZ12}. \label{lem:5.12}
Let $[a, b]$ be a closed finite real interval, $(E, \|\cdot\|)$ a $\mathcal{T}_{\varepsilon, \lambda}-$comp-lete $RN$ module over $K$ with base $(\Omega,\mathcal{F},\mu)$ and $f: [a, b]\rightarrow E$ a Riemann derivable function on $[a, b]$ such that $f^{'}(\cdot): [a, b]\rightarrow E$ is Riemann integrable and such that $$\bigvee\left\{\frac{\|f(t_{2})-f(t_{1})\|}{t_{2}-t_{1}}\left|\right. t_{1}, t_{2}\in [a, b] \mbox{and}~t_{1}\neq t_{2}\right\}\in L_{+}^{0}(\mathcal{F}).$$ Then $$f(b)-f(a)={\int_{a}^{b} f^{'}(t) dt}.$$
\end{lemma}

Now, we can prove Theorem \ref{the:5.10}.

\noindent\textbf{Proof of Theorem \ref{the:5.10}.} Let $\{h_{n}|n\in N\}$ be any sequence in $E$ such that $\{\|h_{n}\|\left|\right. n\in N\}$ converges to $0$ almost everywhere, we only need to prove that $\frac{|f(x_{0}+h_{n})-f(x_{0})-f^{'}(x_{0})(h_{n})|}{\|h_{n}\|}$ converges to $0$ almost everywhere.

Since $f$ is almost everywhere sequently continuously G\^{a}teaux-differentiable at $x_{0}$, there exists $\varepsilon\in L_{++}^{0}(\mathcal{F})$ such that $f^{'}(\cdot): B_{\varepsilon}(x_{0}):=\left\{y\in E\left|\right.\|y-x_{0}\|\leq\varepsilon\right\}\rightarrow E_{c}^{*}$ is almost everywhere sequently continuous at $x_{0}$. We will proceed in two cases as follows.

Case $1$. Let each $h_{n}$ be such that $\|h_{n}\|\leq\varepsilon$ for all $n\in N$. Now, fix $n$, since $x_{0}+th_{n}=(1-t)x_{0}+t(x_{0}+h_{n})\in B_{\varepsilon}(x_{0})$, then the function $H: [0, 1]\rightarrow L^{0}(\mathcal{F})$, defined by $H(t)=f(x_{0}+th_{n})$ for all $t\in [0,1]$, is Riemann derivable $[0, 1]$ and $H^{'}(t)=f^{'}(x_{0}+th_{n})(h_{n})$ for all $t\in [0, 1]$. It is easy to see that $H$ is convex, namely $H(\lambda t_{1}+(1-\lambda)t_{2})\leq \lambda H(t_{1})+(1-\lambda)H(t_{2})$ for all $t_{1}, t_{2}\in [0,1]$ and $\lambda\in[0, 1]$. Similar to the real-valued convex function (see, for instance, \cite{AB06}), it is easy to verify that $H^{'}(\cdot)$ is monotone and $$f^{'}(x_{0})(h_{n})\leq \frac{H(t_{2})-H(t_{1})}{t_{2}-t_{1}}\leq f^{'}(x_{0}+h_{n})(h_{n})$$ for all $t_{1}, t_{2}\in [0, 1]$ such that $t_{1}\neq t_{2}$, so $$\bigvee\left\{\frac{|H(t_{2})-H(t_{1})|}{|t_{2}-t_{1}|}\left|\right.t_{1}, t_{2}\in [0, 1]~\mbox{and}~t_{1}\neq t_{2}\right\}\leq \left|f^{'}(x_{0})(h_{n})\right|\bigvee \left|f^{'}(x_{0}+h_{n})(h_{n})\right|.$$

Since $(E, \|\cdot\|)$ is $\mathcal{T}_{\varepsilon, \lambda}-$complete iff both $E$ is $\mathcal{T}_{c}-$complete and $E$ has the countable concatenation property, which has been proved in \cite{Guo10}, we can employ Lemma \ref{lem:5.12}. Then, by Lemma \ref{lem:5.12}, $$f(x_{0}+h_{n})-f(x_{0})={\int_{0}^{1} f^{'}(x_{0}+th_{n})(h_{n})dt},$$ and hence
\begin{equation*}
\begin{aligned}
\left|f(x_{0}+h_{n})-f(x_{0})-f^{'}(x_{0})(h_{n})\right|
&\leq \int_{0}^{1} \left|f^{'}(x_{0}+th_{n})(h_{n})-f^{'}(x_{0})(h_{n})\right| dt\\
&={\int_{0}^{1} \left(f^{'}(x_{0}+th_{n})(h_{n})-f^{'}(x_{0})(h_{n})\right)dt}\\
&\leq f^{'}(x_{0}+h_{n})(h_{n})-f^{'}(x_{0})(h_{n})\\
&\leq \left\|f^{'}(x_{0}+h_{n})-f^{'}(x_{0})\right\|\|h_{n}\|,
\end{aligned}
\end{equation*}
which shows that $$\frac{\left|f(x_{0}+h_{n})-f(x_{0})-f^{'}(x_{0})(h_{n})\right|}{\|h_{n}\|}\leq \left\|f^{'}(x_{0}+h_{n})-f^{'}(x_{0})\right\|,$$ which converges to $0$ almost everywhere as $n$ tends to $\infty$.

Case $2$. Let $A_{n}=(\|h_{n}\|\leq\varepsilon)$ for each $n\in N$ and define $\overline{h}_{n}=\tilde{I}_{A_{n}}h_{n}+\tilde{I}_{A_{n}^{c}}\cdot 0$ for each $n\in N$. Then $\{\|\overline{h}_{n}\|\left|\right. n\in N\}$ still converges to $0$ almost everywhere and $\|\overline{h}_{n}\|\leq \varepsilon$ for each $n\in N$. By Case $1$ we have proved $\frac{|f(x_{0}+\overline{h}_{n})-f(x_{0})-f^{'}(x_{0})(\overline{h}_{n})|}{\|\overline{h}_{n}\|}$ converges to $0$ almost everywhere as $n \rightarrow \infty$. When one observes that $$\frac{f(x_{0}+h_{n})-f(x_{0})-f^{'}(x_{0})(h_{n})}{\|h_{n}\|}=\frac{f(x_{0}+\overline{h}_{n})-f(x_{0})-f^{'}(x_{0})(\overline{h}_{n})}{\|\overline{h}_{n}\|}$$ on $A_{n}$ for each $n\in N$ (by the local property of $f$), it is also clear that\\ $\frac{|f(x_{0}+h_{n})-f(x_{0})-f^{'}(x_{0})(h_{n})|}{\|h_{n}\|}$ converges to $0$ almost everywhere as $n\rightarrow \infty$ since $A_{n}\uparrow \Omega$.

This completes the proof.  \qed

\section{Subdifferentials and $\varepsilon$-Subdifferentials}
\label{sec:6}

The main results of this section are Theorems \ref{the:6.3} and \ref{the:6.4} below, which are based on the Ekeland's variational principle established by Guo and Yang in \cite{GY12} for an $\bar{L}^{0}-$valued function on complete random metric spaces. Proposition \ref{pro:6.1} below is only a special case of Theorem $3.10$ of \cite{GY12} but it has met the needs of this section.

\begin{proposition} \cite{GY12}. \label{pro:6.1}
Let $(E, \|\cdot\|)$ be a $\mathcal{T}_{c}-$complete $RN$ module over $R$ with base $(\Omega,\mathcal{F},\mu)$ such that $E$ has the countable concatenation property, $\varepsilon\in L_{++}^{0}(\mathcal{F})$, $f: E\rightarrow \bar{L}^{0}(\mathcal{F})$ is local, $\mathcal{T}_{c}-$lower semicontinuous and bounded from below (namely, $\bigwedge \{f(x)|x\in E\}\in L^{0}(\mathcal{F})$), and $x_{0}\in E$ such that $f(x_{0})\leq\bigwedge \{f(x)|x\in E\}+\varepsilon$. Then for each $\lambda\in L_{++}^{0}(\mathcal{F})$ there exists $x_{\lambda}\in E$ such that the following conditions are satisfied:\\
$(1)$ $f(x_{\lambda})\leq f(x_{0})-\frac{\varepsilon}{\lambda}\|x_{\lambda}-x_{0}\|$;\\
$(2)$ $\|x_{\lambda}-x_{0}\|\leq \lambda$;\\
$(3)$ $\frac{\varepsilon}{\lambda}\|x_{\lambda}-x\|+f(x)\nleq f(x_{\lambda})$ for each $x\in E$ such that $x\neq x_{\lambda}$.
\end{proposition}

\begin{remark}\label{rem:6.2}
In fact, $(3)$ of Proposition \ref{pro:6.1} also implies the following relation:\\
$(3)^{'}$ $\frac{\varepsilon}{\lambda}\|x_{\lambda}-x\|+f(x) > f(x_{\lambda})$ for each $x\in E$ such that $x \neq x_{\lambda}$, where ``$>$" means ``$\geq$" and ``$\neq$".\\
If there exists some $v\in E$ with $v \neq x_{\lambda}$ such that $(3)^{'}$ is not true. If $\frac{\varepsilon}{\lambda}\|x_{\lambda}-v\|+f(v)= f(x_{\lambda})$, this contradicts $(3)$. If $\frac{\varepsilon}{\lambda}\|x_{\lambda}-v\|+f(v) \neq f(x_{\lambda})$, then $\mu(A)>0$, where $A=\left(\frac{\varepsilon}{\lambda}\|x_{\lambda}-v\|+f(v)< f(x_{\lambda})\right)$. Let $\bar{v}=\tilde{I}_{A}v+\tilde{I}_{A^{c}}x_{\lambda}$, then $\bar{v}\neq x_{\lambda}$ and $\frac{\varepsilon}{\lambda}\|x_{\lambda}-\bar{v}\|+f(\bar{v}) \leq f(x_{\lambda})$, which contradicts $(3)$, too.
\end{remark}

\begin{theorem}\label{the:6.3}
Let $(E, \|\cdot\|)$, $\varepsilon$, $x_{0}$ and $f$ be the same as in Proposition \ref{pro:6.1}. If, in addition, $f$ is a G\^{a}teaux-differentiable function from $E$ to $L^{0}(\mathcal{F})$, then there exists $x_{\lambda}\in E$ for each $\lambda\in L_{++}^{0}(\mathcal{F})$ such that the following conditions are satisfied:\\
$(1)$ $f(x_{\lambda})\leq f(x_{0})$;\\
$(2)$ $\|x_{\lambda}-x_{0}\|\leq \lambda$;\\
$(3)$ $\|f^{'}(x_{\lambda})\|\leq \frac{\varepsilon}{\lambda}$.
\end{theorem}

\begin{proof}
Let $x_{\lambda}$ be obtained as in Proposition \ref{pro:6.1}, then $x_{\lambda}$ satisfies $(1)$ and $(2)$. By $(3)^{'}$ of Remark \ref{rem:6.2}, for each $t\in [0, 1]$ and $v\in E$ one has that $\frac{\varepsilon}{\lambda}t \|v\|+f(x_{\lambda}+tv)\geq f(x_{\lambda})$, so $-\frac{\varepsilon}{\lambda} \|v\|\leq f^{'}(x_{\lambda})(v)$, which also means $\|f^{'}(x_{\lambda})\|\leq \frac{\varepsilon}{\lambda}$.
\end{proof}

\begin{theorem}\label{the:6.4}
Let $(E, \|\cdot\|)$ be a $\mathcal{T}_{c}-$complete $RN$ module over $R$ with base $(\Omega,\mathcal{F},\mu)$ such that $E$ has the countable concatenation property, $\varepsilon\in L_{++}^{0}(\mathcal{F})$, $f: E\rightarrow \bar{L}^0(\mathcal{F})$ a proper $\mathcal{T}_{c}-$lower semicontinuous $L^{0}-$convex function, $u\in E$ and $u^{*}\in E_{c}^{*}$ such that $u^{*}\in \partial_{\varepsilon}f(u)$. Then there exist $u_{\lambda}\in E$ and $u_{\lambda}^{*}\in E_{c}^{*}$ for each $\lambda \in L_{++}^{0}(\mathcal{F})$ such that the following conditions are satisfied:\\
$(1)$ $\|u-u_{\lambda}\|\leq \lambda$;\\
$(2)$ $\|u^{*}-u_{\lambda}^{*}\|\leq\frac{\varepsilon}{\lambda}$;\\
$(3)$ $u_{\lambda}^{*}\in\partial f(u_{\lambda})$.

In particular, if we take $\lambda=\sqrt{\varepsilon}$, then there exist $u_{\varepsilon}\in E$ and $u_{\varepsilon}^{*}\in E_{c}^{*}$ such that the following conditions are satisfied:\\
$(4)$ $\|u-u_{\varepsilon}\|\leq \sqrt{\varepsilon}$;\\
$(5)$ $\|u^{*}-u_{\varepsilon}^{*}\|\leq\sqrt{\varepsilon}$;\\
$(6)$ $u_{\varepsilon}^{*}\in\partial f(u_{\varepsilon})$.
\end{theorem}

\begin{proof}
Let $G: E\rightarrow \bar{L}^0(\mathcal{F})$ be defined by $G(v)=f(v)-u^{*}(v)+f^{*}(u^{*})$ for all $v\in E$, where $f^{*}: E_{c}^{*}\rightarrow \bar{L}^0(\mathcal{F})$ is the $\mathcal{T}_{c}-$random conjugate function of $f$. Then $G(u)\leq \bigwedge \{G(v)|v\in E\}+\varepsilon$ by the definition of $\partial_{\varepsilon}f(u)$.

We can thus apply Proposition \ref{pro:6.1} to $G$, there exists $u_{\lambda}\in E$ such that the following conditions are satisfied:\\
$(i)$ $\|u-u_{\lambda}\|\leq \lambda$, $G(u_{\lambda})\leq G(u)$;\\
$(ii)$ $\frac{\varepsilon}{\lambda}\|v-u_{\lambda}\|+G(v)\nleq G(u_{\lambda})$ for all $v\in E$ such that $v\neq u_{\lambda}$.

Denote $V (\varepsilon/\lambda)=\left\{(v, a)\in E\times L^0(\mathcal{F})\left|\right. a+\frac{\varepsilon}{\lambda}\|v\|\leq 0\right\}$, then $(ii)$ yields the following relation:\\
$(iii)$ $epi(G)\bigcap \left((u_{\lambda}, G(u_{\lambda}))+V (\varepsilon/\lambda)\right)=\left\{(u_{\lambda}, G(u_{\lambda}))\right\}$.

Indeed, let $(v, r)\in epi(G)$ and $(v-u_{\lambda}, r-G(u_{\lambda}))\in V (\varepsilon/\lambda)$, namely $G(v)\leq r$ and $r-G(u_{\lambda})+\frac{\varepsilon}{\lambda}\|v-u_{\lambda}\|\leq 0$. Then $G(v)+\frac{\varepsilon}{\lambda}\|v-u_{\lambda}\|\leq G(u_{\lambda})$ and $r\leq G(u_{\lambda})$. This shows that $v=u_{\lambda}$ by $(ii)$ and $r=G(u_{\lambda})$.

Further, we also have the following relation:\\
$(iv)$ $\tilde{I}_{A}epi(G)\bigcap \tilde{I}_{A} \left((u_{\lambda}, G(u_{\lambda}))+int(V(\varepsilon/\lambda))\right)=\emptyset$ for all $A\in \mathcal{F}$ with $\mu(A)>0$, where $int(V(\varepsilon/\lambda))$ stands for the $\mathcal{T}_{c}-$interior of $V(\varepsilon/\lambda)$.

In fact, if there exist $A\in \mathcal{F}$ with $\mu(A)>0$, $(v_{1}, r_{1})\in epi(G)$ and $(v_{2}, r_{2})\in (u_{\lambda}, G(u_{\lambda}))+int(V(\varepsilon/\lambda))$ such that $\tilde{I}_{A}(v_{1}, r_{1})=\tilde{I}_{A}(v_{2}, r_{2})$, then, by defining $(v_{3}, r_{3})=\tilde{I}_{A}(v_{1}, r_{1})+\tilde{I}_{A^{c}}(u_{\lambda}, G(u_{\lambda})) (=\tilde{I}_{A}(v_{2}, r_{2})+\tilde{I}_{A^{c}}(u_{\lambda}, G(u_{\lambda})))$, one can see that $(v_{3}, r_{3})\in epi(G)\bigcap\left((u_{\lambda}, G(u_{\lambda}))+ V(\varepsilon/\lambda)\right)$ by the $L^{0}-$convexity of the two sets, so that $(v_{3}, r_{3})=(u_{\lambda}, G(u_{\lambda}))$ by $(iii)$, which further implies that $\tilde{I}_{A}v_{1}=\tilde{I}_{A}v_{2}$ and $\tilde{I}_{A}r_{1}=\tilde{I}_{A}r_{2}=\tilde{I}_{A}G(u_{\lambda})$, whereas $int(V (\varepsilon/\lambda))=\{(v, a)\in E\times L^0(\mathcal{F})\left|\right.a+\frac{\varepsilon}{\lambda}\|v\|<0~\mbox{on}~\Omega\}$ shows that $r_{2}<G(u_{\lambda})$ on $\Omega$ so that it is impossible that $\tilde{I}_{A}r_{2}=\tilde{I}_{A}G(u_{\lambda})$.

We can thus apply Proposition \ref{pro:4.2} to $epi(G)$ and $(u_{\lambda}, G(u_{\lambda}))+int(V(\varepsilon/\lambda))$, there exist $g\in E_{c}^{*}$ and $\beta\in L^0(\mathcal{F})$ such that $g(v_{2})+\beta a_{2}>g(v_{1})+\beta a_{1}$ on $\Omega$ for all $(v_{2}, a_{2})\in epi(G)$ and $(v_{1}, a_{1})\in (u_{\lambda}, G(u_{\lambda}))+int(V(\varepsilon/\lambda))$. Since $a_{2}$ may be arbitrarily large for $(v_{2}, a_{2})$ satisfying $(v_{2}, a_{2})\in epi(G)$, then one must deduce that $\beta>0$ on $\Omega$. Let $h^{*}=g/\beta$, then the following relation is satisfied:\\
$(v)$ $h^{*}(v_{2})+a_{2}\geq h^{*}(v_{1})+a_{1}$ for all $(v_{2}, a_{2})\in epi(G)$ and $(v_{1}, a_{1})\in (u_{\lambda}, G(u_{\lambda}))+V(\varepsilon/\lambda)$ since $int(V(\varepsilon/\lambda))$ is $\mathcal{T}_{c}-$dense in $V(\varepsilon/\lambda)$.

First, by taking $(v_{2}, a_{2})=(u_{\lambda}, G(u_{\lambda}))$ and $(v_{1}, a_{1})=(u_{\lambda}, G(u_{\lambda}))+(w, s)$ in $(v)$ for any given $(w, s)\in V(\varepsilon/\lambda)$, one can deduce that $h^{*}(w)+s\leq0$ for all $(w, s)\in V(\varepsilon/\lambda)$. Further, it is obvious that $(w, s)$ always belongs to $V(\varepsilon/\lambda)$ whenever $s=-\varepsilon/\lambda$ and $w\in E$ is such that $\|w\|\leq1$, and hence $\|h^{*}\|=\bigvee\{h^{*}(w)|w\in E~\mbox{and}~\|w\|\leq1\}\leq\frac{\varepsilon}{\lambda}$. Now, let $u_{\lambda}^{*}=u^{*}-h^{*}$, then $\|u_{\lambda}^{*}-u^{*}\|=\|h^{*}\|\leq \varepsilon/\lambda$.

Then, by taking $(v_{2}, a_{2})=(v, G(v))$ for any give $v\in dom(G)$ and $(v_{1}, a_{1})=(u_{\lambda}, G(u_{\lambda}))$ in $(v)$, one can deduce that $h^{*}(v)+G(v)\geq h^{*}(u_{\lambda})+G(u_{\lambda})$, namely $h^{*}(v-u_{\lambda})+G(v)-G(u_{\lambda})\geq0$ for all $v\in dom(G)$. Again, by the definition of $G$, we have that $u_{\lambda}^{*}(v-u_{\lambda})\leq f(v)-f(u_{\lambda})$ for all $v\in dom(f)$, which clearly implies that $u_{\lambda}^{*}\in \partial f(u_{\lambda})$.

This completes the proof.
\end{proof}

\begin{corollary}\label{cor:6.5}
Let $(E, \|\cdot\|)$ be a $\mathcal{T}_{c}-$complete $RN$ module over $R$ with base $(\Omega,\mathcal{F},\mu)$ such that $E$ has the countable concatenation property and $f: E\rightarrow \bar{L}^0(\mathcal{F})$ a proper $\mathcal{T}_{c}-$lower semicontinuous $L^{0}-$convex function. Then the set $\{u\in E|\partial F(u)\neq\emptyset\}$ is $\mathcal{T}_{c}-$dense in $dom(f)$.
\end{corollary}

\begin{proof}
Since both $E$ and $\mathcal{P}=\{\|\cdot\|\}$ have the countable concatenation property, for any $u\in dom(f)$ and any $\varepsilon\in L_{++}^{0}(\mathcal{F})$, there exists $u^{*}\in \partial _{\varepsilon}f(u)$ by Remark \ref{rem:2.19}. Then, Theorem \ref{the:6.4} produce $u_{\varepsilon}\in E$ and $u_{\varepsilon}^{*}\in E_{c}^{*}$ such that $\|u_{\varepsilon}-u\|\leq \sqrt{\varepsilon}$, $\|u_{\varepsilon}^{*}-u^{*}\|\leq \sqrt{\varepsilon}$ and $u_{\varepsilon}^{*}\in \partial f(u_{\varepsilon})$.
\end{proof}

\begin{remark}\label{rem:6.6}
Although Proposition \ref{pro:4.1} also can deduce Corollary \ref{cor:6.5} since the set $\{u\in E|\partial f(u)\neq\emptyset\}\supset int(dom(f))$ and it is obvious that $int(dom(f))$ is $\mathcal{T}_{c}-$dense in $dom(f)$, we should like to emphasize the power of the Ekeland's variational principle-Proposition \ref{pro:6.1}, since it is Proposition \ref{pro:6.1} that we can obtain a stronger conclusion, namely, $u$ and $u^{*}$ (with $u\in dom (f)$ and $u^{*}\in \partial _{\varepsilon}f(u)$) can be simultaneously approximated by $u_{\varepsilon}$ and $u_{\varepsilon}^{*}$ with $u_{\varepsilon}^{*}\in \partial f(u_{\varepsilon})$, respectively.
\end{remark}

\begin{remark}\label{rem:6.7}
Finally, we should also mention the work of Yang Y. J. in \cite{Yang12}, where she also presented and proved Theorem \ref{the:6.4}, but her proof of $(iv)$ (see the process of the proof of Theorem \ref{the:6.4}) employed the rather complicated technique from the relative topology and the extremely complicated stratification analysis. Compared with hers, our proof of $(iv)$ is straightforward and simple.
\end{remark}

\bibliographystyle{amsplain}

\end{document}